\documentclass{amsart}
\usepackage{amssymb, amsbsy, amsthm, amsmath, amstext, amsopn, verbatim}
\usepackage[all]{xy}
\usepackage{amsfonts}
\usepackage{amscd}
\usepackage{mathrsfs}
\usepackage{verbatim}
\usepackage[usenames]{color}

\newtheorem{thm}{Theorem} [section]
\newtheorem{lemma}[thm]{Lemma}

\newtheorem{corollary}[thm]{Corollary}
\newtheorem{prop}[thm]{Proposition}

\newtheorem*{basic assumption}{Basic Assumption}

\theoremstyle{definition}

\newtheorem*{principal example}{Main Example}

\newtheorem{defn}[thm]{Definition}

\newtheorem{example}[thm]{Example}

\newtheorem{construction}[thm]{Construction}
\newtheorem{assumptions}[thm]{Assumptions}
\newtheorem{assumption}[thm]{Assumption}

\theoremstyle{remark}

\newtheorem{remark}[thm]{Remark}
                         
\newtheorem{claim}[thm]{Claim}

\begin{document}

\numberwithin{equation}{section}

\newcommand{\hs}{\mbox{\hspace{.4em}}}
\newcommand{\ds}{\displaystyle}
\newcommand{\bd}{\begin{displaymath}}
\newcommand{\ed}{\end{displaymath}}
\newcommand{\bcd}{\begin{CD}}
\newcommand{\ecd}{\end{CD}}

\newcommand{\proj}{\operatorname{Proj}}
\newcommand{\bproj}{\underline{\operatorname{Proj}}}
\newcommand{\spec}{\operatorname{Spec}}
\newcommand{\bspec}{\underline{\operatorname{Spec}}}
\newcommand{\pline}{{\mathbf P} ^1}
\newcommand{\pplane}{{\mathbf P}^2}
\newcommand{\coker}{{\operatorname{coker}}}
\newcommand{\ldb}{[[}
\newcommand{\rdb}{]]}

\newcommand{\Sym}{\operatorname{Sym}^{\bullet}}
\newcommand{\Symp}{\operatorname{Sym}}
\newcommand{\Pic}{\operatorname{Pic}}
\newcommand{\AAut}{\operatorname{Aut}}
\newcommand{\PAut}{\operatorname{PAut}}

\newcommand{\too}{\twoheadrightarrow}
\newcommand{\C}{{\mathbf C}}
\newcommand{\cA}{{\mathcal A}}
\newcommand{\cS}{{\mathcal S}}
\newcommand{\cV}{{\mathcal V}}
\newcommand{\cM}{{\mathcal M}}
\newcommand{\bA}{{\mathbf A}}
\newcommand{\cB}{{\mathcal B}}
\newcommand{\cC}{{\mathcal C}}
\newcommand{\cD}{{\mathcal D}}
\newcommand{\D}{{\mathcal D}}
\newcommand{\cs}{{\mathbf C} ^*}
\newcommand{\boldc}{{\mathbf C}}
\newcommand{\cE}{{\mathcal E}}
\newcommand{\cF}{{\mathcal F}}
\newcommand{\cG}{{\mathcal G}}
\newcommand{\G}{{\mathbf G}}
\newcommand{\fg}{{\mathfrak g}}
\newcommand{\bH}{{\mathbf H}}
\newcommand{\cH}{{\mathcal H}}
\newcommand{\cI}{{\mathcal I}}
\newcommand{\cJ}{{\mathcal J}}
\newcommand{\cK}{{\mathcal K}}
\newcommand{\cL}{{\mathcal L}}
\newcommand{\baL}{{\overline{\mathcal L}}}
\newcommand{\M}{{\mathcal M}}
\newcommand{\bM}{{\mathbf M}}
\newcommand{\bm}{{\mathbf m}}
\newcommand{\cN}{{\mathcal N}}
\newcommand{\theo}{\mathcal{O}}
\newcommand{\cP}{{\mathcal P}}
\newcommand{\cR}{{\mathcal R}}
\newcommand{\boldp}{{\mathbf P}}
\newcommand{\boldq}{{\mathbf Q}}
\newcommand{\bbL}{{\mathbf L}}
\newcommand{\cQ}{{\mathcal Q}}
\newcommand{\cO}{{\mathcal O}}
\newcommand{\Oo}{{\mathcal O}}
\newcommand{\OX}{{\Oo_X}}
\newcommand{\OY}{{\Oo_Y}}
\newcommand{\otY}{{\underset{\OY}{\ot}}}
\newcommand{\otX}{{\underset{\OX}{\ot}}}
\newcommand{\cU}{{\mathcal U}}
\newcommand{\cX}{{\mathcal X}}
\newcommand{\cW}{{\mathcal W}}
\newcommand{\boldz}{{\mathbf Z}}
\newcommand{\cZ}{{\mathcal Z}}
\newcommand{\qgr}{\operatorname{qgr}}
\newcommand{\gr}{\operatorname{gr}}
\newcommand{\coh}{\operatorname{coh}}
\newcommand{\End}{\operatorname{End}}
\newcommand{\Hom}{\operatorname{Hom}}
\newcommand{\uHom}{\underline{\operatorname{Hom}}}
\newcommand{\uHomY}{\uHom_{\OY}}
\newcommand{\uHomX}{\uHom_{\OX}}
\newcommand{\Ext}{\operatorname{Ext}}
\newcommand{\bExt}{\operatorname{\bf{Ext}}}
\newcommand{\Tor}{\operatorname{Tor}}

\newcommand{\inv}{^{-1}}
\newcommand{\airtilde}{\widetilde{\hspace{.5em}}}
\newcommand{\airhat}{\widehat{\hspace{.5em}}}
\newcommand{\nt}{^{\circ}}
\newcommand{\del}{\partial}

\newcommand{\supp}{\operatorname{supp}}
\newcommand{\GK}{\operatorname{GK-dim}}
\newcommand{\hd}{\operatorname{hd}}
\newcommand{\id}{\operatorname{id}}
\newcommand{\res}{\operatorname{res}}
\newcommand{\lrar}{\leadsto}
\newcommand{\im}{\operatorname{Im}}
\newcommand{\HH}{\operatorname{H}}
\newcommand{\TF}{\operatorname{TF}}
\newcommand{\Bun}{\operatorname{Bun}}
\newcommand{\Hilb}{\operatorname{Hilb}}
\newcommand{\Fact}{\operatorname{Fact}}
\newcommand{\F}{\mathcal{F}}
\newcommand{\nthord}{^{(n)}}
\newcommand{\Aut}{\underline{\operatorname{Aut}}}
\newcommand{\Gr}{\operatorname{\bf Gr}}
\newcommand{\Fr}{\operatorname{Fr}}
\newcommand{\GL}{\operatorname{GL}}
\newcommand{\gl}{\mathfrak{gl}}
\newcommand{\SL}{\operatorname{SL}}
\newcommand{\ff}{\footnote}
\newcommand{\ot}{\otimes}
\newcommand{\Wx}{\mathcal W_{\mathfrak X}}
\newcommand{\gh}{\text{gr}_{\hbar}}
\newcommand{\ig}{\iota_g}
\def\Ext{\operatorname {Ext}}
\def\Hom{\operatorname {Hom}}
\def\Ind{\operatorname {Ind}}
\def\bbZ{{\mathbb Z}}

\newcommand{\nc}{\newcommand}
\newcommand{\on}{\operatorname}
\nc{\cont}{\on{cont}}
\nc{\rmod}{\on{mod}}
\nc{\Mtil}{\widetilde{M}}
\nc{\wb}{\overline}
\nc{\wt}{\widetilde}
\nc{\wh}{\widehat}
\nc{\sm}{\setminus}
\nc{\mc}{\mathcal}
\nc{\mbb}{\mathbb}
\nc{\Mbar}{\wb{M}}
\nc{\Nbar}{\wb{N}}
\nc{\Mhat}{\wh{M}}
\nc{\pihat}{\wh{\pi}}
\nc{\JYX}{\cJ_{Y\leftarrow X}}
\nc{\phitil}{\wt{\phi}}
\nc{\Qbar}{\wb{Q}}
\nc{\DYX}{\D_{Y\leftarrow X}}
\nc{\DXY}{\D_{X\to Y}}
\nc{\dR}{\stackrel{\bbL}{\underset{\D_X}{\ot}}}
\nc{\Winfi}{\cW_{1+\infty}}
\nc{\K}{{\mc K}}
\nc{\unit}{{\bf \on{unit}}}
\nc{\boxt}{\boxtimes}
\nc{\xarr}{\stackrel{\rightarrow}{x}}
\nc{\Cnatbar}{\overline{C}^{\natural}}
\nc{\oJac}{\overline{\on{Jac}}}
\nc{\gm}{{\mathbf G}_m}
\nc{\Loc}{\on{Loc}}
\nc{\Bm}{\operatorname{Bimod}}
\nc{\lie}{{\mathfrak g}}
\nc{\lb}{{\mathfrak b}}
\nc{\lien}{{\mathfrak n}}
\nc{\E}{\mathcal{E}}

\nc{\Gm}{{\mathbb G}_m}
\nc{\Gabar}{\wb{\G}_a}
\nc{\Gmbar}{\wb{\G}_m}
\nc{\PD}{{\mathbb P}_{\D}}
\nc{\Pbul}{P_{\bullet}}
\nc{\PDl}{{\mathbb P}_{\D(\lambda)}}
\nc{\PLoc}{\mathsf{MLoc}}
\nc{\Tors}{\on{Tors}}
\nc{\PS}{{\mathsf{PS}}}
\nc{\PB}{{\mathsf{MB}}}
\nc{\Pb}{{\underline{\operatorname{MBun}}}}
\nc{\Ht}{\mathsf{H}}
\nc{\bbH}{\mathbb H}
\nc{\gen}{^\circ}
\nc{\Jac}{\operatorname{Jac}}
\nc{\sP}{\mathsf{P}}
\nc{\otc}{^{\otimes c}}
\nc{\Det}{\mathsf{det}}
\nc{\PL}{\on{ML}}

\nc{\ml}{{\mathcal S}}
\nc{\Xc}{X_{\on{con}}}
\nc{\sgood}{\text{strongly good}}
\nc{\Xs}{X_{\on{strcon}}}
\nc{\resol}{\mathfrak{X}}
\nc{\map}{\mathsf{f}}
\nc{\sapc}{\on{sap}_c}
\nc{\base}{Z}
\nc{\bigvar}{\mathsf{W}}
\nc{\alg}{\mathsf{A}}

\title{Derived Equivalence for Quantum Symplectic Resolutions}
\author{Kevin McGerty}
\address{Mathematical Institute\\University of Oxford\\Oxford, England, UK}
\email{mcgerty@maths.ox.ac.uk}
\author{Thomas Nevins}
\address{Department of Mathematics\\University of Illinois at Urbana-Champaign\\Urbana, IL 61801 USA}
\email{nevins@illinois.edu}

\begin{abstract}
Using techniques from the homotopy theory of derived categories and noncommutative algebraic geometry,
we establish a general theory of derived microlocalization for quantum symplectic resolutions.  In particular, our results yield a
new proof of derived Beilinson-Bernstein localization and a derived version of the more recent microlocalization theorems of Gordon-Stafford \cite{GS1,GS2} and Kashiwara-Rouquier \cite{KR} as special cases.  We also deduce a new derived microlocalization result linking cyclotomic rational Cherednik algebras with quantized Hilbert schemes of points on minimal resolutions of cyclic quotient singularities.  

\noindent
\textbf{M.S.C.}: 14H60, 18E30.

\end{abstract}

\maketitle

\section{Introduction}
The Localization Theorem of Beilinson-Bernstein (cf. \cite{BB}) introduced the powerful methods of Hodge theory into the representation theory of complex semisimple Lie algebras, with spectacular consequences.  More recently, there has been progress (cf. \cite{MV}, \cite{GS1}, \cite{GS2}, \cite{KR}, \cite{BKu}, \cite{BLPW})
toward a theory of {\em microlocalization}, i.e. localization on (an open set of) the quantized cotangent bundle,  that would apply to a broader class of interesting and important algebras, including Cherednik-type algebras. 

In the present paper, we establish a  general theory of derived microlocalization for many
quantized symplectic resolutions, as well as more general quantized birational symplectic morphisms.  First, we consider algebras obtained by quantum Hamiltonian reduction from quantizations of smooth affine symplectic varieties with group action.  A principal source of examples of such quantizations comes from rings of differential operators on affine varieties with group action; the resulting Hamiltonian reductions are then rings of
global twisted differential operators on algebraic stacks.  Under reasonable technical 
hypotheses (Assumptions \ref{assumptions}) that hold in examples of interest, we prove that the representation categories of such algebras are derived equivalent to microlocal derived categories on open sets in the symplectic quotient stack, whenever the natural pullback functor is cohomologically bounded.

Our approach results in new proofs of
 known derived equivalences for enveloping algebras \cite{BB, BG} independent of 
the original Beilinson-Bernstein approach.  It also establishes the derived version of microlocalization of 
rational Cherednik algebras of type $A$ \cite{KR} and hypertoric enveloping algebras \cite{BKu}, and yields a new derived microlocalization theorem relating cyclotomic rational Cherednik algebras with quantizations of the Hilbert schemes $(\wt{\C^2/\Gamma})^{[n]}$ of points on minimal resolutions $\wt{\C^2/\Gamma}$ of cyclic quotient singularities.  
Along the way, we develop many basic properties of the microlocal derived categories, including compact generation, indecomposability, and, in the case of quantizations of cotangent bundles, an equivalence between the two natural definitions of
 derived categories.  

\subsection{Microlocalization Theorem}
Suppose $\bigvar$ is a smooth affine complex symplectic variety, equipped with a Hamiltonian action of a connected reductive group $G$ and moment map
$\mu: \bigvar\rightarrow\mathfrak{g}^*$.  A standard source of examples comes from smooth affine varieties $\base$ equipped with $G$-action; then taking
$\bigvar = T^*\base$, the moment map $\mu: T^*\base\rightarrow\mathfrak{g}^*$ is the dual of the infinitesimal $G$-action on $\base$.  
One can associate to the above data 
a naive, usually singular, affine symplectic quotient $X \overset{\on{def}}{=} \mu\inv(0)/\!\!/G = \on{Spec}\C[\mu\inv(0)]^G$.
Alternatively, choosing a character $\chi: G\rightarrow \Gm = \C^*$, one can define (as in Section \ref{classical reduction}) a more refined quotient $\resol = \mu\inv(0)/\!\!/\!_{\chi} G$ by geometric invariant theory (GIT).  There is a natural projective morphism $f:\resol\rightarrow X$.   In many interesting examples, the character $\chi$ can be chosen so that $\resol$
gives a symplectic resolution of the singular symplectic variety $X$ (see the discussion in Section \ref{preliminaries}).  

One can also quantize the situation by replacing functions on $\bigvar$ by a filtered noncommutative algebra $\alg$ whose associated graded algebra 
is $\C[\bigvar]$.  For example, in the case $\bigvar = T^*\base$, the canonical choice of $\alg$ is the ring $\D(\base)$ of differential 
operators on $\base$.  The method of quantum Hamiltonian reduction (Section \ref{quantum reduction})
yields an algebra $U_c$ (depending on a Lie algebra character $c: \mathfrak{g}\rightarrow\C$) that functions as a quantum analog of the algebra $\C[X]$ of functions on $X$.

Similarly, as we explain in Section \ref{microlocalization}, there is a natural quantum analog $D(\cE_{\resol}(c))$ of (the derived category of) $\resol$; and  there are natural quantizations $\mathbb{L}\map^*$, $\mathbb{R}\map_*$ of the inverse and direct images 
$\mathbb{L}f^*$, $\mathbb{R}f_*$ associated to $f:\resol\rightarrow X$. When $\bigvar= T^*\base$ and $G$ acts freely on $\mu\inv(0)^{ss}$ and so $\resol$ is a smooth
symplectic variety, the category $D(\cE_{\resol}(c))$ can be described concretely as a (derived) category of modules over a noncommutative deformation of the sheaf  $\theo_{\resol}$ of functions on $\resol$.   In general, however, we define 
$D(\cE_{\resol}(c))$ via a categorical quotient: the resulting category still possesses those invariants such as characteristic cycles which relate the representation theory of $U_c$ to the geometry of $\resol$, but has the added virtue that it makes sense and has good properties without the assumption that the $G$-action on $\mu^{-1}(0)^{ss}$ is free.

The main result of the paper is the following general criterion for showing that $\mathbb{L}\map^*$ and $\mathbb{R}\map_*$ define mutually quasi-inverse equivalences of derived categories.  The criterion holds under a mild set of hypotheses (Assumptions
\ref{assumptions}) that are satisfied in a broad range of examples of interest.
\begin{thm}[Theorem \ref{main equiv thm} and Corollary \ref{first cor}]\label{main thm intro}
Suppose that $\mu$, $X$, $\resol$, and $f: \resol\rightarrow X$ satisfy Assumptions
\ref{assumptions}. Then the following are equivalent. 
\begin{enumerate}
\item The functor
\bd 
\mathbb{L}\map^*: D(U_c)\rightarrow D(\cE_{\resol}(c))
\ed
 is cohomologically bounded (that is, preserves the class of complexes with cohomologies in only finitely many degrees).
 \item 
 The functor $\mathbb{L}\map^*$ defines an exact equivalence of bounded derived categories.
\end{enumerate}
If $\resol$ is a smooth variety obtained as the quotient of the semistable locus $\mu\inv(0)^{ss}$ by a free $G$-action, then these conditions are also equivalent to:
\begin{enumerate}
\item[(3)] The algebra $U_c$ has finite global dimension.
\end{enumerate}
\end{thm}
\noindent
Since condition (3) always implies condition (1) we obtain:
\begin{corollary}[Corollary \ref{fin gl dim}]
If $U_c$ has finite global dimension, then $\mathbb{L}\map^*$ is an exact equivalence of (dg or triangulated) categories.
\end{corollary}
\noindent
We note that, in fact, abelian equivalences are known in some examples, cf. for example \cite{GS1}, \cite{GS2}, \cite{KR}, \cite{BKu}; related results were also known to K. Kremnizer, I. Grojnowski, and more recently to  Braden-Proudfoot-Webster \cite{BLPW2}.

Theorem \ref{main thm intro} has the following immediate consequence:
\begin{corollary}
\label{derivedequivalence drho}
Suppose that $U_c$ and $U_{c+d\rho}$ both have finite global dimension, where $\rho: G\rightarrow \Gm$ is a character.  Then $D(U_c)$ and $D(U_{c+d\rho})$ are equivalent.
\end{corollary}
\noindent
In particular, the corollary provides derived equivalences between the algebras $U_c$ even when the natural shift functors ``cross walls.''  We also remark that the theorem does not depend on the group character $\chi$---in particular, if derived microlocalization holds for one such $\chi$ then it holds for any other.  

The conditions of the theorem are known to be satisfied in many instances.  In Section \ref{wreath}, we describe an interesting class of examples, namely the spherical Cherednik algebras associated to wreath products of cyclic groups.  Let $\mu_\ell\cong \Gamma \subset SL(2)$ denote a cyclic subgroup.  The symmetric group $S_n$ acts by 
permutations on the product $\Gamma^n$, and the semidirect product group is the {\em wreath product}, 
denoted by $S_n \wr \mu_\ell$.  By construction, the wreath product is realized as a group of symplectic 
reflections (in the $n$-fold product of $\C^2$); the associated spherical rational Cherednik algebra $U_{k, c}$ 
depends on parameters $k\in\C$ and $c\in\C^{\ell-1}$.  In \cite{Gordon, EGGO}, this algebra is realized via quantum
Hamiltonian reduction.  As we explain in Section \ref{wreath}, combining the constructions of \cite{Gordon, EGGO} and the result of 
\cite{DG} on the {\em aspherical values} of the parameters gives the following result:
\begin{corollary}
For the spherical Cherednik algebra $U_{k,c}$ associated to the wreath product  $S_n \wr \mu_\ell$, the functor $D(U_{k,c})\rightarrow D(\cE_{\resol}(k,c))$ is an exact equivalence of triangulated categories 
away from a finite collection of hyperplanes described by Proposition \ref{cyclotomic global dimension}.
\end{corollary}
\noindent
An explicit combinatorial description of these hyperplanes is given in Section \ref{wreath}; since the notation is a bit involved we refer the reader to that section for details.

Note that for these examples, Corollary \ref{derivedequivalence drho} can be viewed as a spherical analogue of the derived equivalences established by Gordon and Losev \cite[\S 5.5]{GL} for the full rational Cherednik algebra. Their equivalences are established by completely different means -- namely the quantisation of the  exceptional objects constructed by \cite{BK} (therein called ``weakly Procesi bundles''). In \cite{GL}, such derived equivalences are used in a detailed study of the category $\mathcal O$ representations of these rational Cherednik algebras, and in particular to establish cases of a conjecture \cite[\S 5]{R} of Rouquier.

Theorem \ref{main thm intro} also generalizes to non-affine situations in which a good quotient (in the GIT sense) exists; see Section \ref{good quotient}, and especially Theorem \ref{second thm intro}, for more details.  Here is a sample application.  It is shown in \cite{FG} how to realize Etingof's \cite{Et1} type $A$ Cherednik algebra of an algebraic curve $C$ by quantum Hamiltonian reduction.  Derived microlocalization then follows from Theorem \ref{second thm intro} for the type $A$ Cherednik algebra of an arbitrary smooth $C$ whenever the sheaf of spherical algebras $U_c$ on $\on{Sym}^n(C)$ is locally of finite global dimension.  Using \'etale charts on $C$, one reduces to $C=\mathbb{A}^1$ to show that this happens for the ``usual'' values of the parameter:

\begin{corollary}\label{Ch alg of curve}
Let $C$ be a smooth algebraic curve.  Let $U_c$ denote the sheaf (on $\on{Sym}^n(C)$) of spherical subalgebras of the type $A$ rational Cherednik 
algebra of $C$ as defined in \cite{Et1}.  
Let $D\big(\cE_{(T^*C)^{[n]}}(c)\big)$ denote the associated microlocal category.  Then 
\bd
D\big(\cE_{(T^*C)^{[n]}}(c)\big)\simeq D(U_c) 
\;\;
\text{provided}
\;\;
c\notin\{-\frac{p}{q} \in \mathbb{Q}\; |\; 1\leq p < q\leq n\}.
\ed 
\end{corollary}

An interesting class of examples comes from the {\em Mori dream spaces} \cite{HK}, a class of varieties that includes flag varieties, spherical varieties, and toric varieties as special cases.  The main constructions of the paper can be extended to the setting needed for derived (micro)localization for Mori dream spaces.  
As a consequence, one can reprove derived localization \cite{BG} for flag varieties independently of the classical method of Beilinson-Bernstein, and in particular using only the maximal torus $H$ and not the full Borel action. This example also lends particular interest to condition (1) of Theorem \ref{main thm intro}, since it is shown in \cite{BMR} that it is quite easy to check that $\mathbb{L}\map^*$ is cohomologically bounded (for twists by regular central characters)---significantly
 easier than to prove that the corresponding central reduction of the enveloping algebra has finite global dimension.  
Thus, as we explain in Section \ref{flag varieties}, our approach
also gives a new proof for finite global dimension of the central reductions of the enveloping algebras of semisimple Lie algebras (albeit one that is certainly not simpler than the approach via \cite{BB}).  We discuss only the flag variety case in Section \ref{flag varieties}; we expect to return to the more general subject of Mori dream spaces elsewhere. 

Finally, in Section \ref{DQ} we briefly discuss how one can extend our main theorem to the context of deformation quantizations of arbitrary symplectic resolutions.

\subsection{Methods}
Our approach to Theorem \ref{main thm intro} is, in spirit, similar to the derived Beilinson-Bernstein
localization theorem of \cite{BMR}.  Namely, it is easy to see that the induction functor $\mathbb{L}\map^*$ is faithful, with left quasi-inverse given by its right adjoint $\mathbb{R}\map_*$.  The trick is then to see that the kernel of $\mathbb{R}\map_*$ and image of $\mathbb{L}\map^*$ give an orthogonal decomposition of 
$D(\cE_{\resol}(c))$, and furthermore that the latter category is indecomposable; this would give the equivalence.  Unfortunately, we cannot see how to adapt the method of \cite{BMR}, using the Calabi-Yau property and two-variable Serre duality, directly to establish these properties: \cite{BMR} relied heavily on the Azumaya property, i.e. the large centers of the noncommutative rings appearing.  

Instead, we take something of a detour through more 
homotopical methods, establishing as in \cite{Neeman} the existence of a {\em right} adjoint $\map^!$ to 
$\mathbb{R}\map_*$ and proving that it commutes with colimits.\footnote{The statement of Theorem \ref{main thm intro} itself
is partly inspired by \cite{LN}.}  We then compute the right adjoint on some objects using a more classical version of (noncommutative) Serre-Grothendieck--type duality (Section \ref{NC duality section}).  Although this version of duality may be a consequence of the powerful framework developed by Yekutieli-Zhang (see \cite{YZ-differential} and references there), we develop it by using \v{C}ech methods to reduce to the commutative case; this has the additional advantage of proving quasi-properness of the functor $\mathbb{R}\map_*$.  
Once we have computed $\map^!$ on enough objects, it follows that $\mathbb{L}\map^* = \map^!$, and 
a small modification of  the argument of \cite{BMR} finishes the proof.  

Along the way, we establish a number of fundamental properties of the microlocal category 
$D(\cE_{\resol}(c))$.  For example, we prove that $D(\cE_{\resol}(c))$ is compactly generated 
(Proposition \ref{compact generation}) and indecomposable (Proposition \ref{indecomposable}), and that, in the differential operator case, the 
two natural definitions of $D(\cE_{\resol}(c))$, as the derived category of a corresponding abelian category 
and as a category of weakly equivariant complexes that are only homotopically strongly equivariant,
are equivalent (Proposition \ref{two derived cats agree}).  Although many of these properties play important roles in our proofs, they should also be understood as establishing basic features of the microlocal derived categories that make clear that they are suitable objects for geometric study.  An appendix (Section \ref{proof of serre}) develops the basic tools of \v{C}ech resolutions that we use several times in the body of the paper.

\subsection{Acknowledgments}
The authors are indebted to David Ben-Zvi for extensive conversations, and for patiently answering numerous questions about compactly generated categories; and to Toby Stafford for teaching them about noncommutative \v{C}ech complexes.  The authors are grateful to Gwyn Bellamy, Will Donovan, Iain Gordon, Ivan Losev, Travis Schedler, Susan Sierra, Ben Webster, and Michael Wemyss for very helpful comments and suggestions.  

The first author was supported by a Royal Society research fellowship.  The second author was 
supported by NSF grant DMS-0757987, NSA grant H98230-12-1-0216, and NSF grant DMS-1159468.

\section{Equivalences of Derived Categories}
\label{equiv der cat}
In this section, we lay out a general framework (following the technique of Bridgeland as applied in, for example, \cite{BMR}) for proving derived equivalences.  We also summarize some basics of compactly generated categories.  Whenever we write the word ``colimit'' in a derived category we mean homotopy colimit.

\subsection{General Techniques for Derived Equivalence}
We begin with some useful lemmas.
\begin{lemma}[cf. Lemma 2.7 of \cite{BNa}]\label{cg lemma}
Suppose that $F: C\leftrightarrows D: G$ form an adjoint pair of functors.
\begin{enumerate}
\item If $G$ preserves colimits (i.e. is {\em continuous}) and $d\in D$ is compact then $F(d)$ is compact.
\item If $G$ is faithful and $S$ is a set of generators of $G$ then $F(S)$ is a set of generators of $C$. 
\end{enumerate}
\end{lemma}

\begin{lemma}[\cite{Lu}, Lemma 4.7.2]\label{compact = perfect}
Let $A$ be a (possibly noncommutative) $\C$-algebra.  
The compact objects of $D(A)$ are the perfect objects, i.e. the objects quasi-isomorphic to bounded complexes of finitely generated projective left
$A$-modules.  Moreover, $D(A)$ is compactly generated, and every object is a small colimit of compact objects.
\end{lemma}

\begin{lemma}\label{ind-category}
If $C$ is a compactly generated cocomplete dg category with a set $S$ of compact generators, and $\wt{S}$ denotes the smallest full triangulated category containing $S$, then any object of $C$ is a colimit of a diagram in $\wt{S}$. 

\end{lemma}
\begin{proof}
This is a version of Lemma 2.2.1 of \cite{SS}. See also Theorem 2.1 of \cite{Neeman}.
\end{proof}

Let $C, D$ be triangulated categories and 
\bd
F: C \leftrightarrows D: G
\ed
an adjoint pair of exact functors.  Consider the following conditions:
\begin{enumerate}
\item[(C1)] $D$ is compactly generated.
\item[(C2)] $G$ preserves coproducts.
\item[(C3)] $G$ takes a set $S$ of compact generators of $D$ to compact objects of $C$.
\item[(C4)] There is a set $S$ of compact generators of $C$ for which $F(U) \cong F^!(U)$ for all $U\in S$.
\item[(C5)] The adjunction $1_C\rightarrow G\circ F$ is an isomorphism.
\item[(C6)] $D$ is indecomposable.  
\end{enumerate}

\begin{lemma}[cf. \cite{BMR}, Lemma 3.5.2]\label{BMR lemma precursor}
Let $F: C\rightarrow D$ be a faithful exact functor of triangulated categories with a right adjoint $G: D\rightarrow C$ satisfying (C5) and also
\begin{enumerate}
\item the functor $G$ has a right adjoint $H: C\rightarrow D$;
\item the functors $F$ and $H$ have the same essential image; and
\item the category $D$ is indecomposable.
\end{enumerate}
Then $F, G,$ and $H$ are equivalences.
\end{lemma}

\begin{prop}\label{BMR lemma}
Suppose that $C, D$ are triangulated categories with an adjoint pair $F:C\leftrightarrows D: G$ of exact functors.  
\begin{enumerate}
\item If these data satisfy conditions (C1)-(C3) above, then $G$ has an exact right adjoint $F^!: C\rightarrow D$ that preserves colimits.  
\item If these data in addition satisfy conditions (C4)-(C6), then $F, G$ are mutually quasi-inverse equivalences of categories.
\end{enumerate}
\end{prop}
\begin{proof}
Conditions (C1) and (C2) imply $F^!$ exists by Theorem 4.1 of \cite{Neeman}.  Condition (C3) implies that $F^!$ preserves coproducts by Theorem 5.1 of \cite{Neeman}; then $F^!$ preserves colimits by:
\begin{lemma}[Lemma 4.1 of \cite{BN}]\label{colimits}
If $F^!$ preserves coproducts then it preserves all small colimits.
\end{lemma}  
To prove the equivalence statement, we apply Lemma \ref{BMR lemma precursor};
by this lemma (and conditions (C5) and (C6)), the final claim of Proposition \ref{BMR lemma} will follow once we show that $F$ and $F^!$ have the same essential image.  But this follows from condition (C4) by Lemma \ref{ind-category}.  This completes the proof of Proposition \ref{BMR lemma}.
\end{proof}

\section{Group Actions and Hamiltonian Reduction}\label{preliminaries}
\begin{basic assumption}
Throughout the paper, we assume all varieties $\bigvar$ and groups $G$ are connected.
\end{basic assumption}

\subsection{Conventions on Group Actions}
Suppose a group $G$ acts on a smooth variety $\bigvar$.  Given a character $\chi: G\rightarrow \Gm$, we make the trivial line bundle 
$\mathsf{L} = \bigvar\times {\mathbb A}^1$ into a $G$-equivariant line bundle via $g\cdot (x, z) = (g\cdot x, \chi(g)z)$.  
We will use the following conventions.  

Suppose that $\bigvar$ is affine.  For $f\in \C[\bigvar]$, $g\in G$, we let
$(g\cdot f)(x) = f(g\inv x)$.  If $\bigvar = T^*\base$, and $f\in \C[\base]$, $\theta\in \D(\base)$, and $g\in G$, we let
$(g\cdot \theta)(f) = g\cdot (\theta(g\inv\cdot f))$.  

Recall that a function 
$f: \bigvar\rightarrow {\mathbb A}^1$ is a {\em relative invariant} or {\em semi-invariant} of weight $\chi$ if $f(g\cdot x) = \chi(g) f(x)$ for all $g\in G$ 
and $x\in \bigvar$.  Suppose that $F: \bigvar\rightarrow \mathsf{L} = \bigvar\times \mathbb{A}^1$ is a section, and write 
$F(x) = (x, f(x))$ for a function $f: \bigvar\rightarrow \mathbb{A}^1$.  Then 
$g\cdot F(x) = (gx, \chi(g)f(x))$, and so $F$ is $G$-equivariant if and only if $f$ is $\chi$-semi-invariant. 

  We write $\C[\bigvar]^{G,\chi}$ for the $\chi$-semi-invariant functions on $\bigvar$.
\begin{example}
Let $\bigvar=\C^{n+1}$ with $\Gm$ acting via the vector space structure.  Letting $\chi(w) = w^\ell$, we find that a $\chi$-semi-invariant is a 
homogeneous polynomial of degree $\ell$, i.e. a section of $\theo(\ell)$ on  ${\mathbb P}^n$.
\end{example}

\subsection{Classical Hamiltonian Reduction}\label{classical reduction}
Suppose that $G$ acts on a smooth affine symplectic variety $\bigvar$ with symplectic form $\Omega$.   We assume that $\bigvar$ is equipped with a $G$-equivariant moment map $\mu: \bigvar \rightarrow \mathfrak{g}^*$.  

\begin{principal example}
Our principal example throughout the paper will be the following.  Suppose $G$ acts on a smooth affine variety $\base$.  
The $G$-action induces an infinitesimal $\mathfrak{g}$-action $\mu: \mathfrak{g}\rightarrow H^0(T_{\base})$, 
which can be interpreted as a vector space map $\mathfrak{g}\rightarrow \C[T^*\base]$.  The latter map thus defines a map of varieties (which we denote by the same symbol)
$\mu: T^*\base\rightarrow \mathfrak{g}^*$, the {\em classical moment map} for the group action.  
\end{principal example}

Given a smooth affine symplectic $G$-variety $\bigvar$ with moment map $\mu$, write 
\bd
X = \mu^{-1}(0)/\!\!/G = \on{Spec} \C[\mu^{-1}(0)]^G
\ed for the affine quotient variety.  This is a (typically singular) algebraic variety.

We will proceed to choose an additional character $\chi: G\rightarrow \Gm$ of the group $G$.\footnote{Later, when doing quantum Hamiltonian reduction, we will choose a Lie algebra character $c:\mathfrak{g}\rightarrow\C$; these choices need not be related.}  
 As in \cite{King}, this yields a variety
\bd
\resol = \mu^{-1}(0)/\!\!/_{\chi}G = \on{Proj}\Big(\bigoplus_{\ell\geq 0} \C[\mu^{-1}(0)]^{G,\chi^\ell}\Big).
\ed
The variety $\resol$ comes equipped with a projective morphism $f: \resol\rightarrow X$.  We make the following assumptions:
\begin{assumptions}\label{assumptions}
\mbox{}
\begin{enumerate}
\item The moment map $\mu$ is flat.
\item[(2a)] The affine quotient $X = \mu\inv(0)/\!\!/G$ is Gorenstein with trivial dualizing sheaf 
$\omega_X\cong \theo_X$.
\item[(2b)] The GIT quotient $\resol = \mu\inv(0)/\!\!/\!_{\chi} G$ is Gorenstein with trivial dualizing
sheaf $\omega_{\resol} \cong \theo_{\resol}$.
\item[(3)] Under the morphism $f: \resol\rightarrow X$, we have $\mathbb{R}f_*\theo_{\resol} = \theo_X$.
\end{enumerate}
\end{assumptions}
\noindent
By Grauert-Riemenschneider, if $\resol$ is smooth, $X$ is normal, and $f$ is birational, then (2a) and (2b) together imply (3).
In the next sections, we discuss conditions under which the assumptions hold.

\subsection{Flatness of the Moment Map}
We discuss flatness of $\mu$ for $\bigvar = T^*\base$.

\begin{lemma}\label{goodness lemma}
Suppose that  $\bigvar = T^*\base$ where $\base$ is a smooth variety with $G$-action and the $G$-action on $\bigvar$ is the one induced from $\base$.
Let $\mu: \bigvar\rightarrow \lie^*$ denote the moment map and $N = \mu^{-1}(0)$.  
Then the following conditions are equivalent:
\begin{enumerate}
\item The map $\mu$ is flat and dominant.
\item $N$ is a $G$-equivariant complete intersection in $\bigvar$ of dimension $\on{dim}(\bigvar)-\on{dim}(G)$.
\item $N$ has dimension $2\on{dim}(\base)-\on{dim}(G)$.
\item Letting $[\base/G]$ denote the quotient stack, $\on{dim}(T^*[\base/G]) = 2\on{dim}([\base/G])$.  
\item One has $\on{codim}\{y\in \base \; |\; \on{dim}(G_y) = n\} \geq n$ for all $n\geq 1$.
\end{enumerate}
\end{lemma}

\begin{defn}
[Beilinson-Drinfeld, Page 6 of \cite{BD}] We call $[\base/G]$ satisfying the conditions of Lemma \ref{goodness lemma}  {\em good} or {\em good for lazybones}.  
\end{defn}

\begin{example}
Associated to any finite subgroup $\Gamma\subset SL(2)$ and any $n\geq 1$ there is a representation $\base= \mathcal{M}(\Gamma, n)$ of a reductive group $G = G(\Gamma, n)$ \cite{GG}.  It is shown in \cite{GG} that the moment map for this $\bigvar = T^*\base$ and $G$ is
always flat.  
\end{example}
Flatness of the moment map in a general quiver setting is characterized in \cite{CB-flat}.

\subsection{Gorenstein Property and Rational Singularities}
Assumption \ref{assumptions}(3) is a weak version of $X$ having rational singularities: one could say $X$ ``has rational singularities 
with respect to $\resol$.''  Indeed:
\begin{remark}
If $\resol$ is smooth and $f$ is birational then Assumption \ref{assumptions}(3) exactly says
that $X$ has rational singularities.  Note that it implies, in particular, that $X$ is normal: if not, then
the map $f$ would factor through the normalization $\wt{X}$ of $X$, and the composite 
$\theo_X\rightarrow \theo_{\wt{X}}\rightarrow f_*\theo_{\resol}$ would split $\theo_X\rightarrow\theo_{\wt{X}}$, which cannot happen because $\wt{X}$ is integral.
\end{remark}

Assumptions \ref{assumptions}(2a), (2b), (3) are understood in the case when $f$ is a symplectic resolution of a singular symplectic variety (\cite{B}, Definition 1.1), as the following shows.
\begin{lemma}
\label{sympresol}
If $f\colon \resol \to X$ is a symplectic resolution, then \ref{assumptions}(2a),(2b),(3) hold.
\end{lemma}

\begin{proof}
(2b) holds by hypothesis.
By assumption, $X$ is a symplectic variety, which by Proposition 1.3 of \cite{B} implies it is Gorenstein with rational singularities, so (3) holds.  
In particular, $X$ is normal, so the top exterior power of the symplectic form on $X^{sm}$ extends to a nonvanishing regular section of $\omega_X$, proving (2a).
\end{proof}

\subsection{Reduction for Tori}
\begin{lemma}
If $G$ is a torus, then, replacing $G$ with its image in $\on{Aut}(\bigvar)$, Assumptions \ref{assumptions}(1) and (3) hold.  If $\bigvar$ is a vector space, $G$ is a torus, and the action is linear and unimodular (see \cite{BKu} for definitions), then for a generic choice of character, Assumption \ref{assumptions}(2) holds.
\end{lemma}
\begin{proof}
Suppose $G$ is a torus.  Then, because subgroups of $G$ do not come in continuous families, the stabilizer of a generic point of $\base$ equals the kernel of $G\rightarrow \on{Aut}(\base)$.  Then by Remark 9.2(5) and Proposition 9.4(2) of \cite{Sch}, after replacing $G$ by its image
in $\on{Aut}(\base)$ the moment map becomes flat.
Henceforth, we may assume that the moment map is flat.  

The statement about Assumption \ref{assumptions}(2) is Proposition 4.11 of \cite{BKu}.

Assuming that the moment map is flat, pulling back the Koszul resolution $K(\mathfrak{g})\rightarrow \C$ of the trivial $\C[\mathfrak{g}^*]$-module to $\C[\bigvar]$ gives a
free equivariant resolution of $\C[\mu\inv(0)]$ with terms $\C[\bigvar]\otimes\wedge^k(\mathfrak{g}^*)$; since $G$ is abelian these are even equivariantly free.  Note now that $\mu\inv(0)^{ss} = \bigvar^{ss}\cap\mu\inv(0)$.  It follows that, letting $j: \bigvar^{ss}\hookrightarrow \bigvar$ denote the inclusion, we have that
$(\mathbb{R}j_*\theo_{\mu\inv(0)^{ss}})^G \simeq \big(\mathbb{R}j_*\theo_{\bigvar^{ss}}\otimes K(\mathfrak{g})\big)^G$.  By \cite{T}, one has  
\bd
\big(\mathbb{R}j_*\theo_{\bigvar^{ss}}\otimes K(\mathfrak{g})\big)^G \simeq \big(\theo_{\bigvar}\otimes K(\mathfrak{g})\big)^G \cong \C[\mu\inv(0)]^G.  
\ed
This establishes (3).   
 \end{proof}

\subsection{Quantum Hamiltonian Reduction}\label{quantum reduction}
Suppose again that $G$ acts on a smooth affine symplectic variety $\bigvar$ with moment map $\mu$.   

\vspace{.4em}

\begin{center}
{\em We suppose throughout the remainder of the 
paper that Assumptions \ref{assumptions} hold.}  
\end{center}

\vspace{.4em}

We assume that we are given a {\em filtered quantization} $(\alg, F, \mu, \phi)$ of our geometric setting.  By this we mean the following.
\begin{enumerate}
\item $\alg$ is a nonnegatively filtered noncommutative $\C$-algebra whose associated graded $\on{gr}\,\alg = \on{gr}_F\,\alg$ comes equipped with an isomorphism $\phi: \on{gr}\,\alg\rightarrow \C[\bigvar]$ of Poisson algebras.
\item $\alg$ comes equipped with a rational left $G$-action that preserves the filtration $F$ and makes the isomorphism $\phi: \on{gr}\,\alg \rightarrow\C[\bigvar]$ $G$-equivariant.
\item $\alg$ comes equipped with a choice of integer $n\geq 1$ and a  $G$-equivariant Lie algebra homomorphism $\mu: \mathfrak{g}\rightarrow \alg_n = F_n\,\alg$ whose associated graded 
$\mu: \mathfrak{g}\rightarrow \on{gr}_n\alg$ equals the moment map for the $G$-action on $\bigvar$.
\item Letting $\alpha: \mathfrak{g}\rightarrow \on{End}_{\C}(\alg)$ denote the derivative of the $G$-action on $\alg$, we have
$[\mu(z), \theta] = \alpha(z)(\theta)$ for all $z\in\mathfrak{g}, \theta\in\alg$.
\end{enumerate}
In particular, the isomorphism $\phi$ equips $\C[\bigvar]$ with a nonnegative $G$-equivariant grading.  The grading induces a $\Gm$-action on $\bigvar$ that commutes with the $G$-action and makes $\mu$ into a $\Gm$-equivariant moment map for $\bigvar$ (provided one uses the weight $n$ action of $\Gm$ on $\mathfrak{g}^*$).   We refer to $\mu: \mathfrak{g}\rightarrow\alg_n$, and the induced map (which we also denote by $\mu$) $\mu: U(\mathfrak{g})\rightarrow \alg$, as {\em quantum (co)moment maps}. 

Since, by assumption, the ring $\alg$ is nonnegatively filtered, it is {\em Zariskian} in the terminology of \cite{LvO}.  In particular, the Rees ring $\cR(\alg) =  \oplus_{k\in{\mathbb Z}} F_k(\alg)t^k$ is (left and right) noetherian.  The Rees module of a filtered $\alg$-module $M$ is 
$\cR(M) = \oplus_{k\in{\mathbb Z}} M_kt^k \subseteq M[t,t^{-1}]$.  Note that $\cR(\alg)$ is naturally a $\C[t]$-algebra and $\cR(M)$ is a $\C[t]$-module. As usual, $\cR(M)/t\cR(M)$ is canonically isomorphic to $\on{gr}_F(M)$ as a $\cR(\alg)/t\cR(\alg) = \on{gr}_F(\alg)$-module, and, for any $a\neq 0$, $\cR(M)/(t-a)\cR(M) = M$ as a $\cR(\alg)/(t-a)\cR(\alg) = \alg$-module.  We assume all $\cR = \cR(\alg)$-modules are graded $\cR$-modules.

One has the following basic facts about filtered $\alg$-modules.
\begin{prop}\label{filtrations and ass gr}
\mbox{}
\begin{enumerate}
\item\label{gr M=0 implies M=0} If $M$ is an $\alg$-module with good filtration and $\on{gr}(M) = 0$ then $M=0$.
\item Suppose that $M$ is equipped with a separated and exhaustive filtration $F$ (that is, $\cap_k F_k(M) = 0$ and $\cup_k F_k(M) = M$).  If $\on{gr}_F(M)$ is finitely generated, then $F$ is a good filtration, and, in particular, $\cR(M)$ is a finitely generated $\cR(\alg)$-module.
\end{enumerate}
\end{prop}
\begin{proof}
(1) is \cite[Lemma~1.2]{LVO}.  (2) is \cite[Theorem~I.5.7]{LvO}.
\end{proof}

\begin{principal example}
Suppose $G$ acts on a smooth affine variety $\base$.  
Let $\bigvar = T^*\base$ and $\alg = \D(\base)$, the algebra of differential operators on $\base$.  
The infinitesimal $\mathfrak{g}$-action $\mu: \mathfrak{g}\rightarrow H^0(T_\base)$ can be interpreted as a map $\mathfrak{g}\rightarrow \D^1(\base)$ of Lie algebras.  The map $\mu$ thus induces a quantum (co)moment map $\mu: U({\mathfrak g}) \rightarrow \D(\base)$.  Both algebras are naturally filtered (we use the filtration by order of differential operators on $\D(\base)$) and the associated graded map agrees with the classical moment map defined above. 
\end{principal example}

\begin{construction}
Given a Lie algebra character $c: \mathfrak{g}\rightarrow \C$, we define a {\em twisted quantum (co)moment map} $\mu_c: \mathfrak{g}\rightarrow \alg_1\subset \alg$ by 
$\mu_c(z) = \mu(z) + c(z)$ for $z\in \mathfrak{g}$.  
\end{construction}

\begin{assumption}
Suppose from now on that $\bigvar$ is smooth and affine and that $G$ is reductive.
\end{assumption}
We form the {\em quantum Hamiltonian reduction} $U_c$ of $\alg$ as follows.  
Define
\bd
M_c \overset{\on{def}}{=} \alg/\alg\mu_c(\mathfrak{g}) \;\; \text{and} \;\; M_c^\dagger \overset{\on{def}}{=} \mu_c(\mathfrak{g})\alg\backslash\alg.
\ed
Taking $G$-invariants gives algebras
\bd
U_c \overset{\on{def}}{=} M_c^G, \hspace{2em} U_c^{\dagger} \overset{\on{def}}{=} (M_c^\dagger)^G.
\ed
Note that, although $M_c$ is only a left $\alg$-module and not an algebra, the submodule $\left(\alg\mu_c(\mathfrak{g})\right)^G$ is actually a two-sided ideal in $\alg^G$ and so $U_c$ is an algebra. 
Moreover, then $M_c$ is a $(\alg, U_c)$-bimodule and $M_c^\dagger$ is a $(U_c^\dagger,\alg)$-bimodule.

\begin{remark}
As explained in Sections 1.7 and 1.8 of \cite{BB}, 
in the case $\bigvar= T^*\base$, $\alg=\D(\base)$, 
the algebra $U_c$ can be interpreted as the algebra of global sections of a $D$-algebra, or more precisely an algebra of twisted differential operators, on the quotient stack $\base/G$.  Under this interpretation, one has $U_c = H^0\big(\D_{\base/G}(\theo(c))\big)$.
\end{remark}

\begin{example}
Let $\base=\C^{n+1}$ and $G=\Gm$ acting with weight one.  Let $\bigvar = T^*\base$ and $\alg = \D(\base)$.  Let $\chi(w) = w$, so $\theo(\chi)$ corresponds to $\theo(1)$ on
${\mathbb P}^n \subset \base/G$.  Then $M_c = \D/\D(e+c)$ where $e = -\sum x_i\partial_{x_i}$ is the Euler operator for the $\Gm$-action. 
If $f$ is a homogeneous polynomial of degree $d$, then $e\cdot f = -d\cdot f$, so $(e+d)f=0$.  Thus $M_{d}$ is a $\D$-module through which the
natural $\D$-action on sections of $\theo(d)$ factors, and it turns out that $H^0(\D_{X/\Gm}(\theo(\chi)^{d})) \cong \D_{{\mathbb P}^n}(\theo(d))$.
\end{example}

The filtration $F$ on $\alg$ induces a filtration on the subalgebra
$\alg^G$ of $G$-invariants and thus also a filtration on the quotient $U_c$.  Since $G$ is reductive we have
\bd
\on{gr}(\alg^G) \cong \C[\bigvar]^G.
\ed

\begin{lemma}\label{Gorenstein}
The algebra $U_c$ is Auslander Gorenstein with rigid Auslander dualizing complex $\mathbb{D}^\bullet = U_c$.  The Rees algebra $\cR(U_c)$ is
Auslander Gorenstein with rigid Auslander dualizing complex $\mathbb{D}_{\cR(U_c)} = \cR(U_c)$.
\end{lemma}
\begin{proof}
Since $X$ is Gorenstein with trivial dualizing sheaf (Assumption \ref{assumptions}(2a)), Theorem 3.9 of \cite{Bj} implies that $U_c$ is Auslander-Gorenstein; 
Theorem 5.3 of \cite{Ek} implies that $\cR(U_c)$ is Auslander Gorenstein.
Example 2.3(a) of \cite{YZ} then gives the statements about dualizing complexes.  
\end{proof}
We write 
\bd
\mathbb{D}_{\alg}(-) = \mathbb{R}\Hom_{\alg}(-, \alg)\;\;\; \text{and} \;\;\; \mathbb{D}_{U_c}(-) = \mathbb{R}\Hom_{U_c}(-, U_c)
\ed
for the (Verdier) dualizing functors.  

As we have remarked, by construction, the left $\alg$-module $M_c$ defined above is naturally a $(\alg, U_c)$-bimodule.
Since $\mu$ is flat by assumption, Proposition 2.3.12 of \cite{Bjbook} yields that $\alg$ is a flat $U(\mathfrak{g})$-module via $\mu_c$.  Thus, we get a free $\alg$-resolution of $M_c$ as $\alg\otimes_{U(\mathfrak{g})} C(\mathfrak{g})$, where $C(\mathfrak{g}) = U(\mathfrak{g})\otimes \bigwedge^{\bullet}(\mathfrak{g})$ is the Chevalley-Eilenberg resolution of the trivial one-dimensional $U(\mathfrak{g})$-module.

\begin{lemma}\label{M_c dual}
We have $M_c^\dagger = \mathbb{D}_{\alg}(M_c)[g]$, the (shifted) Verdier dual of $M_c$ as a left $\alg$-module, where
$g=\on{dim}(G)$.
\end{lemma}
\noindent
The proof is a calculation using the Chevalley-Eilenberg complex.

Since $\mathfrak{g}$ is reductive and $G$ is connected, as in (7.4) of \cite{GG}, we have
\bd
\on{Hom}_{\alg}(M_c,M_c) = (M_c)^{\mu_c(\mathfrak{g})} = M_c^G = U_c,
\ed
and similarly $\on{Hom}_{\alg}(M_c^{\dagger}, M_c^{\dagger}) = U_c^\dagger$.  
Lemma \ref{M_c dual} thus implies 
$U_c \cong U_c^\dagger$.

The filtrations on $M_c$ and $M_c^\dagger$ induced from $\alg$ are compatible with the filtrations on both $\alg$ and $U_c$, making $M_c$ and $M_c^\dagger$ filtered bimodules.  
\begin{prop}
\mbox{}
\begin{enumerate}
\item With respect to the filtrations defined above, we have $\on{gr}(M_c) \cong \C[\mu^{-1}(0)] \cong \on{gr}(M_c^\dagger)$ as a $(\on{gr}\,\alg, \on{gr}(U_c))$-bimodule, respectively $(\on{gr}(U_c), \on{gr}\,\alg)$-bimodule.
\item $M_c$ and $M_c^\dagger$ are Cohen-Macaulay $U_c$-modules.
\end{enumerate}
\end{prop}
\begin{proof}
The first part is Proposition 2.4 of \cite{Holland} (\cite{Holland} is written in a $\D$-module context but the proof works equally well in our more general setting).  The second part follows (as on page 268 of \cite{EG}) from Cohen-Macaulayness of
$\C[\mu\inv(0)]$, which follows from Lemma \ref{goodness lemma}.
\end{proof}

\section{Microlocal Derived Categories}\label{section 4}
In this section we describe the (microlocal) categories of interest to us and establish some of their basic properties.  The book \cite{KS} is an excellent general reference for background on relevant topics.  Given an abelian category $\cC$, we let $D(\cC)$ denote its unbounded derived category: a good discussion can be found in \cite{BN}. 

We assume throughout this section that $\bigvar$ is a smooth, affine symplectic $G$-variety (with $G$ connected reductive) with quantization $\alg$ as in Section 
\ref{quantum reduction}.

An $\alg$-module $M$ is {\em weakly $G$-equivariant} if:
\begin{enumerate}
\item  $M$ is equipped with the structure of a {\em rational} $G$-representation, that is, it is a union of its finite-dimensional sub-representations;
 and 
 \item $g\cdot(\theta m) = (g\cdot \theta)(g\cdot m)$ for all $m\in M$, $g\in G$, and $\theta\in \alg$. 
 \end{enumerate}
It is immediate from the definition that every weakly $G$-equivariant $\alg$-module is the union of its weakly equivariant submodules that are finitely generated as $\alg$-modules.
 
 Let $(\alg,G)-\on{mod}$ denote the abelian category of weakly $G$-equivariant left $\alg$-modules with $\alg$-module homomorphisms respecting the $G$-actions.  We also let $D(\alg,G)$ denote the (unbounded) derived category of weakly $G$-equivariant left $\alg$-modules.  The category $(\alg,G)-\on{mod}$ is a Grothendieck category and thus has enough injectives (cf. Remark 2.5 of \cite{VdB}).  
\begin{lemma}\label{enough proj}
For any finite-dimensional representation $V$ of $G$, the $\alg$-module $\alg\otimes V$ is a projective object of $(\alg,G)-\on{mod}$.  In particular,  $(\alg,G)-\on{mod}$ has enough projectives.
\end{lemma}

\subsection{Twisted Equivariant Categories}
Suppose that $M$ is a weakly $G$-equivariant $\alg$-module.  Differentiating the $G$-action on $M$ gives an infinitesimal $\mathfrak{g}$-action on $M$, i.e. a Lie algebra homomorphism $\alpha: \mathfrak{g}\rightarrow \on{End}_{\C}(M)$.  For $z\in\mathfrak{g}$, $a\in\alg$ and $m\in M$, it satisfies $\alpha(z)(am) = [\mu(z),a]\cdot m + a\cdot\alpha(z)(m)$ (the Leibniz rule).  Given a Lie algebra character $c:\mathfrak{g}\rightarrow \C$ and the twisted moment map $\mu_c$ defined above, one then defines
\bd
\gamma_{M, c}: \mathfrak{g}\rightarrow \on{End}_{\alg}(M), \hspace{2em} z\mapsto \alpha(z) - \mu_c(z).
\ed

We say that a weakly $G$-equivariant $\alg$-module is
{\em $c$-twisted $G$-equivariant} or {\em $(G,c)$-equivariant} if $\gamma_{M,c} = 0$; that is, if $\gamma_{M,0}$ agrees with the composite 
\bd
\mathfrak{g}\xrightarrow{-c} \C \xrightarrow{w\mapsto w\cdot \on{Id}_M} \on{End}_{\alg}(M).
\ed
We write $(\alg, G, c)-\on{mod}$ for the abelian category of $(G,c)$-equivariant $\alg$-modules.  This category is a Grothendieck category and thus has enough
injectives (cf. also Corollary 2.8 and Lemma 1.5.3 of \cite{VdB}).  We remark that our definitions of weakly and strongly equivariant $\alg$-modules agree with the standard ones when $\alg=\D(\base)$.

We will also need $(G,c)$-equivariant {\em right} $\alg$-modules; we briefly describe those now.  Let $G$ act on
$\alg$ on the right by the dual action.  The infinitesimal right action is then given by $\mu^r = -\mu: \mathfrak{g}\rightarrow \alg$.  A right $\alg$-module $M$ equipped with a rational right $G$-action is {\em weakly $G$-equivariant} if $(m\cdot\theta)\cdot g = (m\cdot g)(\theta\cdot g)$ for all $g\in G, \theta\in \alg, m\in M$.  Given a weakly equivariant right module $M$ we let $\alpha^r$ denote the derivative of the right $G$-action (the superscript $r$ is used to emphasize that it is a right action).  We say a weakly equivariant right $\alg$-module $M$ is {\em $(G,c)$-equivariant} if the endomorphism
$\gamma_{M,c}^r = \alpha^r - (\mu^r -c) \equiv 0$; note the change in sign from $\gamma_{M,c} = \alpha - (\mu+c)$ for left modules above.  

 Recall the $(\alg, U_c)$-bimodule 
$M_c = \alg/\alg \mu_c(\mathfrak{g})$  and its dual $M_c^\dagger$ defined above. 

\begin{lemma}\label{twisted equivariant modules}
The modules $M_c$ and $M_c^\dagger$ are $(G,c)$-equivariant.
\end{lemma}

\begin{lemma}[cf. Section 3 of \cite{Kashiwara}]\label{adjunction}
The natural forgetful functor 
\bd
\sapc: (\alg, G, c)-\on{mod}\xrightarrow{\text{forget}} (\alg,G)-\on{mod}
\ed
 has a left adjoint $\Phi_c$ defined by 
 \begin{equation}
 \ds\Phi_c(M) = M/\big(\sum_{z\in\mathfrak{g}}\gamma_{M,c}(z)M\big).
 \end{equation}
   Similarly, the forgetful functor 
   \bd
   \sapc: D((\alg, G,c)-\on{mod})\rightarrow D((\alg, G)-\on{mod})
   \ed
    has left adjoint $\mathbb{L}\Phi_c$.  
\end{lemma}

We next introduce some specific $(G,c)$-equivariant modules; these are a special case of the induction functor described in Section
\ref{adjunctions}.  Let $\rho: G\rightarrow \Gm$ be a character.  We get a $(G,c)$-equivariant module 
$M_c(\rho) \overset{\on{def}}{=} \Phi_c(\D\otimes\rho)$, where $\D\otimes\rho$ denotes $\alg$ considered as a weakly $G$-equivariant $\alg$-module by twisting the $G$-action by $\rho$.  A calculation shows that 
$\gamma_{\alg\otimes\rho, c} = \gamma_{\alg, c-d\rho}$ where $d\rho: \mathfrak{g}\rightarrow \C$ denotes the derivative of $\rho$.  
Hence 
\begin{equation}\label{twist calc}
M_c(\rho) \cong M_{c-d\rho}\otimes\rho
\end{equation}
as $(\alg,G)$-modules.

\subsection{Quantum Hamiltonian Reduction Functor}
As above, we let $(\alg,G)-\on{mod}$ denote the abelian category of weakly $G$-equivariant $\alg$-modules and $(\alg, G, c)-\on{mod}$ denote the full subcategory of $(G,c)$-equivariant modules.

We define the functor of {\em quantum Hamiltonian reduction} 
\bd
\mathbb{H}: (\alg,G)-\on{mod} \longrightarrow U_c-\on{mod}
\ed
 by the formula 
\begin{equation}\label{defn of H}
\mathbb{H}(M) = \on{Hom}_{\alg}(M_c, M)^G = \Hom_{(\alg,G)}(M_c, M).
\end{equation}
Composing with $\sapc$ defines a functor
\bd
\mathbb{H}_c = \mathbb{H}\circ\sapc: (\alg,G,c)-\on{mod}\rightarrow U_c-\on{mod}.
\ed
If $M$ is $(G,c)$-equivariant, then it is completely reducible as a $\mathfrak{g}$-module under the action
via $\mu_c$.  Thus,
since $\mathfrak{g}$ is reductive, as in (7.3) and (7.4) of \cite{GG}, we have
\begin{equation}\label{invariants and coinvariants}
\on{Hom}_{\alg}(M_c,M) = M^{\mu_c(\mathfrak{g})} = M_{\mu_c(\mathfrak{g})} = M_c^\dagger \otimes_{\alg} M
\end{equation}
for $M\in (\alg,G,c)-\on{mod}$. 
Since $G$ is reductive, we thus have:
\begin{lemma}\label{basics of H}
\mbox{}
\begin{enumerate}
\item $\mathbb{H}_c$ is an exact functor.
\item If $G$ is connected and $M$ is $(G,c)$-equivariant, then 
\bd
\mathbb{H}_c(M) = \Hom_{\alg}(M_c,M) = M^G.
\ed
\item In particular, $\mathbb{H}_c(M_c) = U_c$.
\end{enumerate}
\end{lemma}
In general,  we have 
\bd
\mathbb{R}\Hom_{\alg}(M_c, M) = M_c^\dagger[g]\overset{\mathbb{L}}{\otimes}M.
\ed

\subsection{Microlocal Abelian Categories}\label{microlocalization}
Suppose $M$ is an $\alg$-module.  If $M$ is finitely generated over $\alg$, one may choose a good filtration and then 
$SS(M) = \on{supp}\, \on{gr}(M)\subseteq \bigvar$ is independent of the choice of good filtration.  For an arbitrary $\alg$-module $M$, we let $SS(M)$ denote the union of $SS(M')$ over finitely generated submodules $M'\subseteq M$.  

If $M$ is weakly $G$-equivariant, then $SS(M)$ is a union of $G$-stable closed subsets of $\bigvar$, hence itself $G$-stable.  
Fixing, as before, a character $\chi$, we let $\bigvar^{uns}$ denote the $\chi$-unstable subset of $\bigvar$, and let $(\alg,G)-\on{mod}^{uns}$ denote the full subcategory of modules with singular support contained in $\bigvar^{uns}$, and similarly for $(\alg,G,c)-\on{mod}$.  
We make similar definitions for $\cR = \cR(\alg)$-modules: namely, if $M$ is a graded $\cR$-module, then $M/tM$ is a 
$\on{gr}\,\cR = \C[\bigvar]$-module.  If $M$ is finitely generated, we write $SS(M)$ for the support of $M/tM$, and in general let $SS(M)$ denote the union of $SS(M')$ over finitely generated $\cR$-submodules $M'\subseteq M$.   We then define $(\cR,G)-\on{mod}^{uns}$ and $(\cR,G,c)-\on{mod}^{uns}$ analogously.
All of these categories
are {\em localizing subcategories} in the sense of Section III.4.4 of \cite{Popescu}, (see also III.1 of \cite{Gabriel} and Exercise 8.13 of \cite{KS}).  In particular:
\begin{lemma}\label{quotient cats}
\mbox{}
\begin{enumerate}
\item The quotient abelian categories 
\bd
\E_{\resol}(c)-\on{mod} \overset{\on{def}}{=} (\alg,G,c)-\on{mod}/(\alg,G,c)-\on{mod}^{uns}, \;
(\alg,G)-\on{mod}^{ss} \overset{\on{def}}{=} (\alg,G)-\on{mod}/(\alg,G)-\on{mod}^{uns}
\ed 
and
\bd
(\cR,G,c)-\on{mod}/(\cR,G,c)-\on{mod}^{uns}, \;\;\; (\cR,G)-\on{mod}/(\cR,G)-\on{mod}^{uns}
\ed
exist and are Grothendieck categories, hence they have enough injectives.  
\item The projection functors on the quotient categories in (1) have right adjoints.
\item For every $a\in \C$, the functor $\C[t]/(t-a)\otimes_{\C[t]} -$, from (weakly or strongly equivariant) $\cR$-modules to $\alg$-modules (for $a\neq 0$) or $\C[\bigvar]$-modules (for $a=0$), descends to the quotient categories compatibly with the projections.  
\end{enumerate}
\end{lemma}
We will abusively write $\Hom{\E_{\resol}(c)}(-,-)$ to mean the Hom in the category $\E_{\resol}(c)-\on{mod}$.
Letting 
\bd
\pi_c: (\alg,G,c)-\on{mod}\longrightarrow \E_{\resol}(c)-\on{mod}, \;\; \pi: (\alg,G)-\on{mod}\longrightarrow (\alg,G)-\on{mod}^{ss}
\ed
denote the projections, we let
\bd
\Gamma_c: \E_{\resol}(c)-\on{mod} \rightarrow (\alg,G,c)-\on{mod}, \;\; 
\Gamma: (\alg,G)-\on{mod}^{ss} \rightarrow (\alg,G)-\on{mod}
\ed
denote the right adjoints.  We use the same notation when we replace $\alg$ by $\cR$.

We write $\E_{\resol}(c) = \pi_c(M_c)$ and $\E_{\resol}(c)\otimes\rho \overset{\on{def}}{=} \pi_c(M_c(\rho))$ for a character $\rho: G\rightarrow \Gm$.

\subsection{Microlocal Derived Categories}\label{microlocal derived cats}
We define $D(\cE_{\resol}(c))$ to be the unbounded derived category of the abelian category 
$\cE_{\resol}(c)-\on{mod}$.   We will refer to this category as the {\em microlocal derived category}. 
We refer to \cite{BN} for basics of unbounded derived categories.  The following Lemma computes $\mathbb R\Gamma_c(\E_{\resol}(c))$ in terms of a \v Cech complex. The machinery of noncommutative \v Cech complexes which we need is gathered in the appendix, Section \ref{proof of serre}.

\begin{lemma}\label{Gamma of D}
We have 
\bd
\mathbb{R}\Gamma_c(\E_{\resol}(c)) = \check{C}^\bullet(M_c) \simeq \check{C}^\bullet(\alg)\otimes_{U(\mathfrak{g})} C(\mathfrak{g}),
\ed
where $\check{C}^\bullet$ denotes the \v{C}ech complex from Section \ref{proof of serre}.
\end{lemma}
\begin{proof}
We have $\E_{\resol}(c) = \pi_c(M_c)$ by definition.  We thus have 
$\mathbb{R}\Gamma_c(\E_{\resol}(c)) \simeq \check{C}^\bullet(M_c)$
 from Theorem \ref{Cech calculates}.  Now the definition of $M_c$ and the flatness of $\alg$ over $U(\mathfrak{g})$ implies (via Proposition \ref{Cech of a quotient}) that
 \bd
 H^i\big(\check{C}^\bullet(M_c)\big) \cong H^i\big(\check{C}^\bullet(\alg)\big)\otimes_{U(\mathfrak{g})} C(\mathfrak{g}), 
 \ed
 where (as above) $C(\mathfrak{g})$ is the Chevalley-Eilenberg complex.
\end{proof}

We note here that there is an alternative construction of the (dg enhanced if one prefers) derived category $D(\E_{\resol}(c))$ using a quotient operation for (the dg enhancement of) the derived category $D((\alg,G,c)-\on{mod})$.  Namely, we let $D((\alg,G,c)-\on{mod})^{uns}$ denote the full dg subcategory consisting
of objects whose cohomologies have unstable support in the above sense.  The associated full triangulated subcategory is thick.  Passing to the quotient category, which we also denote by 
$D(\E_{\resol}(c))$, we obtain a dg enhanced triangulated category together with a cohomology functor 
$H^0: D(\E_{\resol}(c))\rightarrow \E_{\resol}(c)-\on{mod}$ (see \cite{Drinfeld}, especially Appendix A).  

The fact that the two constructions agree follows from the next lemma.
\begin{lemma}\label{abelian vs dg}
Suppose that $D=D(A)$ is the (bounded or unbounded) derived category of an abelian category $A$, $T\subset A$ is a localizing subcategory, and $D(A/T)$ is the (corresponding bounded or unbounded, according to $D(A)$) derived category of $A/T$.  Let 
$C$ denote the full subcategory of $D$ consisting of objects whose cohomologies (with respect to the standard $t$-structure) lie in $T$.  Then the canonical map $D(A)\rightarrow D(A/T)$ induces an equivalence $D(A)/C\simeq D(A/T)$.
\end{lemma}

\begin{remark}\label{duality descends}
The Verdier duality functor $M\mapsto \mathbb{D}_{\alg}(M) = \mathbb{R}\Hom_{\alg}(M,\alg)$ preserves singular support
by Proposition D.4.2 of \cite{HTT} (the proposition immediately generalizes to our setting).  It follows that Verdier duality descends to the microlocal derived category; we denote this induced functor still by $\mathbb{D}$.
\end{remark}

\subsection{Sheaf-Theoretic Interpretation of Derived Categories}\label{sheaf-theoretic section}
The contents of Section \ref{sheaf-theoretic section} are for motivation only and will not be used elsewhere in the paper.

In the case when $G$ acts freely on $\mu\inv(0)^{ss} = \mu\inv(0)^s$, yielding a smooth quotient variety,
the category $D(\E_{\resol}(c))$ has a geometric interpretation which has appeared elsewhere in the literature (and explains our notation).
Namely, the microrestriction of the sheaf $M_c$ of $(\alg,G)$-modules to $\mu\inv(0)^{s}$ descends (via $G$-invariant push forward) to a sheaf
of filtered algebras $\cA_{\resol}(c)$ on $\resol = \mu\inv(0)^s/G$, with associated graded sheaf $\theo_{\resol}$. In the (related but different) context of $\mathscr W$-algebras, this is the approach taken by Kashiwara and Rouquier, \cite[\S2.5]{KR}.   The algebra $\cA_{\resol}(c)$ 
is an analog of the sheaf of microdifferential operators on the cotangent bundle of a smooth variety.  The abelian category $\E_{\resol}(c)-\on{mod}$ comes equipped with a restriction functor to
$\cA_{\resol}(c)-\on{mod}$, and similarly $D(\E_{\resol}(c))$ comes with a functor to $D(\cA_{\resol}(c))$.  The restriction functor is faithful, and thus $\E_{\resol}(c)-\on{mod}$ can be understood as a category of modules for the sheaf of rings $\cA_{\resol}(c)$.  

In more detail, recall that we may construct a sheaf of rings $\cE_{\bigvar}$ on $\bigvar$ by microlocalizing $\alg$ (see for example \cite{AVV} for an algebraic description of this construction), and we may similarly microlocalize $M_c$ to obtain an $\cE_{\bigvar}$-module $\cM_c$. 
Although $A$ is nonnegatively filtered, its microlocalization $\cE_{\bigvar}$ is a $\mathbb Z$-filtered sheaf of algebras (elements with symbol of positive degree can be inverted), which thus contains a natural sheaf of subalgebras $\cE_{\bigvar}(0)$ consisting of the sections in the nonpositive part of the filtration. A \textit{lattice} for a coherent $\cE_{\bigvar}$-module $\mathcal N$ is a coherent $\cE_{\bigvar}(0)$-submodule which generates $\mathcal N$ as a $\cE_{\bigvar}$-module. The data of an $\cE_{\bigvar}(0)$-lattice is the microlocal analogue of a good filtration. See sections $7.5$ and $8.7$ of \cite{Kashiwara1} for more details in the setting of microdifferential operators.

Using the sheaf of modules $\cM_c$ we may define a sheaf-theoretic version of quantum Hamiltonian reduction: The $G$-invariant direct image of $\mathcal End(\cM_c)$ to $\resol$ gives a sheaf of filtered algebras $\cA_{\resol}(c)$ on $\resol$. Moreover, the module $\cM_c$ has a natural $\cE_{\resol}(0)$-lattice $\cM_c(0)$ and the $G$-invariant direct image of its $\cE_{\bigvar}(0)$-endomorphisms equips $\cA_{\resol}$ with a corresponding sheaf of subalgebras $\cA_{\resol}(c)(0)$.
  An $\cA_{\resol}(c)$-module $M$ is {\em good} if it admits a coherent $\cA_{\resol}(c)(0)$-submodule $M(0)$ which generates $M$ as an $\cA_{\resol}(c)$-module.
\begin{prop}
Suppose the action of $G$ on $\mu^{-1}(0)^{ss}$ is free, so that we have a sheaf of algebras $\cA_{\resol}(c)$ on $\resol$ as described above. 
There is a natural functor from $\cE_{\resol}(c)-\on{mod}$ to $\cA_{\resol}(c)-\on{mod}$. The essential image in $\cA_{\resol}(c)-\on{mod}$ of the category of finitely generated $\cE_{\resol}(c)$-modules is the subcategory of good modules; moreover, the functor to this category is full.  The essential image of $\cE_{\resol}(c)-\on{mod}$ is the ind-category of the category of good modules, i.e. the category of modules that are colimits of systems of good modules. In particular we see that in this case
\[
\cE_{\resol}(c)-\on{mod} \cong \on{Ind}(\cA_{\resol}(c)-\on{mod}_{\on{good}}).
\]
\end{prop}
A proof of the analogous result for DQ algebras will appear in a forthcoming joint paper with Bellamy and Dodd.  
A similar statement appears, with a completely different proof strategy, in the comprehensive recent paper \cite{BLPW2}.

\subsection{On Different Types of Derived Categories}\label{different types}
Section \ref{different types} is not needed in the rest of the paper: we include it for completeness only.

In general, if $S$ is a stack then the ``correct" derived category of $\D$-modules on $S$ is not equivalent to the derived category of the abelian category of $\D$-modules (see Section 7.6 of \cite{BD} or \cite{BL}).  In general, one should proceed as follows.  Suppose $N^\bullet$ is a
complex of weakly $G$-equivariant $\D(W)$-modules (with $G$-equivariant $\D(W)$-module homomorphisms).  A {\em $(G,c)$-equivariant $\D$-complex} is such a weakly $G$-equivariant $N^\bullet$ together with a morphism $\mathfrak{g}\otimes N^\bullet \xrightarrow{i_-} N^\bullet[-1]$,
$Z\otimes n\mapsto i_Z(n)$, such that for any $Z\in \mathfrak{g}$ one has 
\bd
i_Z^2 = 0 \;\;\; \text{and} \;\;\; di_Z + i_Zd = \gamma_{N^\bullet, c}
\ed 
(cf. Definition 7.6.11 of \cite{BD}).  
One lets $C(\D(W), G, c)$ denote the dg category of such complexes.  There is a natural cohomology functor $C(\D(W), G, c)\xrightarrow{H^0}
(\D(W), G,c)-\on{mod}$.  Localizing $C(\D(W), G, c)$ with respect to quasi-isomorphisms gives a pre-triangulated
dg category which we denote by $D_h(\D(W),G,c)$.  The heart of the standard $t$-structure on $D_h(\D(W),G,c)$ is exactly $(\D(W), G,c)-\on{mod}$
(Section 7.6.11 of \cite{BD}).  One can then, as above, take the quotient by the thick subcategory of complexes whose cohomologies have unstable support to obtain a ``homotopical'' derived category, which we will denote by
$D_h(\cE_{\resol}(c))$. 

In general, this construction will not agree with the one we gave above.  
However, in our case:
\begin{prop}\label{two derived cats agree}
Under Assumption \ref{assumptions}(1), i.e. flatness of the moment map $\mu$, the induced functors 
\bd
D((\D,G,c)-\on{mod})\rightarrow D_h(\D, G,c), \hspace{2em} D(\E_{\resol}(c)-\on{mod})\rightarrow D_h(\cE_{\resol}(c))
\ed
 are exact equivalences of triangulated categories.
\end{prop}

The proof reduces to (a twisted version of) Theorem 3.4 of \cite{BL}: Assumption \ref{assumptions}(1) guarantees that the hypotheses of that theorem
are satisfied. Since we will not use this result elsewhere in the paper, we omit the proof.

\subsection{Adjunctions and Compact Generation}\label{adjunctions}
In this section we prove that the microlocal derived category is compactly generated.

We begin with a general fact.
\begin{lemma}\label{weakly cg}
Suppose that $A$ is a (possibly noncommutative) $\C$-algebra with rational $G$-action, where $G$ is a reductive group.  Then
$D((A,G)-\on{mod})$ is compactly generated by the objects $V\otimes A$ for finite-dimensional $G$-representations $V$.
\end{lemma}

As in Lemma \ref{adjunction}, there is an adjoint pair
\bd
\mathbb{L}\Phi_c: D((\alg, G)-\on{mod}) \leftrightarrows D((\alg,G,c)-\on{mod}): \sapc,
\ed
where $\mathbb{L}\Phi_c$ exists by Lemma \ref{enough proj}.
As in Section \ref{microlocalization}, there is an adjoint pair
\bd
\pi_c: D((\alg,G,c)-\on{mod})\leftrightarrows D(\E_{\resol}(c)): \mathbb{R}\Gamma_c.
\ed
\begin{prop}\label{compact generation}
\mbox{}
\begin{enumerate}
\item The right adjoints $\sapc$ and $\mathbb{R}\Gamma_c$ above preserve colimits.
\item The category $D(\E_{\resol}(c))$ is compactly generated.
\end{enumerate}
\end{prop}
\begin{proof}
The only part of (1) that requires proof is that $\mathbb{R}\Gamma_c$ preserves colimits.  To see this, first observe that 
$(\alg,G)-\on{mod}$ is locally noetherian, generated by the noetherian objects $\alg\otimes V$ for finite-dimensional representations $V$ of $G$.  Using the adjunction $(\Phi_c, \sapc)$, one finds that the objects $\Phi_c(\alg\otimes V)$ are noetherian and generate 
$(\alg, G, c)-\on{mod}$, so the latter category is locally noetherian.  
Part (1) then follows from Corollaire 1, Section III.4 of \cite{Gabriel} and \cite[Proposition 15.3.3]{KS}.

For part (2), apply Lemma \ref{cg lemma} to the adjoint pair $(\pi_c\circ \mathbb{L}\Phi_c, \sapc\circ\mathbb{R}\Gamma_c)$, using part (1) for continuity of the right adjoint and Lemma \ref{weakly cg} for compact generation of $D((\alg,G)-\on{mod})$.
\end{proof}

We will be interested below in some particular compact objects, namely the ones 
$\E_{\resol}(c)\otimes\rho\overset{\on{def}}{=} \pi_c\, M_c(\rho)$ induced from line bundles.  Note that, by Lemma \ref{enough proj}, $\alg\otimes\rho$ is projective in $(\alg,G)-\on{mod}$; hence 
\begin{equation}\label{underived equals derived}
\pi_c\, M_c(\rho) = \pi_c\Phi_c(\alg\otimes\rho) = \pi_c\circ\mathbb{L}\Phi_c(\alg\otimes\rho).
\end{equation}

\begin{lemma}\label{computing twisted}
For any character $\rho:G\rightarrow\Gm$ and any $M\in D(\E_{\resol}(c))$, we have 
\bd
\mathbb{R}\Hom_{\E_{\resol}(c)}(\E_{\resol}(c)\otimes\rho, M) \cong \Hom_G\big(\rho, \mathbb{R}\Gamma_c(M)\big).
\ed
\end{lemma}
\begin{proof}
Using the adjunctions described above, we have
\begin{multline*}
\mathbb{R}\Hom_{\E_{\resol}(c)}(\E_{\resol}(c)\otimes\rho, M) 
= \mathbb{R}\Hom(\pi_c\circ\mathbb{L}\Phi_c(\alg\otimes\rho), M)\\
 \cong
 \mathbb{R}\Hom_{(\alg, G)}\big(\alg\otimes\rho, \mathbb{R}\Gamma_c(M)\big)
 = \Hom_G\big(\rho, \mathbb{R}\Gamma_c(M)\big)
 \end{multline*}
 as desired, where the first equality follows from \eqref{underived equals derived}.
 \end{proof}

\subsection{Indecomposability of $D(\E_{\resol}(c))$}

\begin{prop}\label{indecomposable}
The microlocal derived category $D(\E_{\resol}(c))$ and the bounded subcategory 
$D^b(\E_{\resol}(c))$ are
 indecomposable.
\end{prop}
\begin{proof}
Note that by Lemma \ref{cg lemma} and part $(1)$ of Proposition \ref{compact generation}, the functors $\pi_c$ and $\mathbb L\Phi_c$ both take compact objects to compact objects. It is then easy to see that that the Proposition follows if we can find a collection $S$ of compact objects of $D((\alg,G)-\on{mod})$  for which
\begin{enumerate}
\item[(a)] For each $s\in S$, the object $\pi_c\circ \mathbb{L}\Phi_c(s)$ is indecomposable and lies in $D^b(\E_{\resol}(c))$.
\item[(b)] For all $t, t'\in \pi_c\circ \mathbb{L}\Phi_c(S)$, there is some $i\in\mathbb{Z}$ for which 
\bd
\on{Hom}^i_{D(\E_{\resol}(c))}(t,t') \neq 0 \;\;\; \text{or} \;\;\; \on{Hom}^i_{D(\E_{\resol}(c))}(t',t) \neq 0.
\ed
\item[(c)] The set $\pi_c\circ \mathbb{L}\Phi_c(S)$ generates $D(\E_{\resol}(c))$.
\end{enumerate}
Consider the object $L =\alg\otimes \chi^\ell$ given by twisting the trivial module by a power of the character $\chi$.  The induced $\E_{\resol}(c)$-module is $\E_{\resol}(c)\otimes \chi^\ell$; by \eqref{twist calc} this is isomorphic to $\pi_c M_{c-\ell d\chi}\otimes \chi^\ell$, hence,
by Lemma \ref{computing twisted} and Lemma \ref{basics of H}(3),
$\End_{\cE_\resol(c)}(\E_{\resol}(c)\otimes \chi^\ell) \cong U_{c-\ell d\chi}$.  Since $\on{gr}\, U_{c-\ell d\chi} \cong \C[X]$, the algebra 
$U_{c-\ell d\chi}$ is a domain.  But this implies that $\cE_\resol(c)\otimes \chi^\ell$ is indecomposable: a nontrivial direct sum decomposition of $\cE_\resol(c)\otimes\chi^\ell$ would yield a pair of (nonzero) endomorphisms, namely the projections on the two direct factors, whose composite is zero, which cannot happen if the endomorphism ring is a domain.

Let $\mathscr{L}$ denote the line bundle on $\resol$ associated to the character $\chi$.  Choose an integer $N$ for which $\mathscr{L}^N$ is very ample, and let $S = \{\alg\otimes \chi^{aN}[n] \; |\; a, n\in \mathbb{Z}\}$.  By the previous paragraph, $S$ satisfies requirement (a). 
For requirement (b), we use \v{C}ech resolutions as in Section \ref{proof of serre}.  Observe:

\bd
\begin{aligned}
\mathbb{R}\Hom\big(\E_{\resol}(c)\otimes\chi^{aN}, \E_{\resol}(c)\otimes\chi^{bN}\big)
= & \; \Hom_G\big(\chi^{aN}, \mathbb{R}\Gamma_c\big(\pi_c(M_c(\chi^{bN}))\big)\big)&\text{by Lemma \ref{computing twisted}}\\
= & \; \Hom_G\big(\chi^{aN}, \check{C}^\bullet(M_c(\chi^{bN}))\big)  &\text{by Theorem \ref{Cech calculates}}.\\
\end{aligned}
\ed
Giving $M_c(\chi^{bN})$ the filtration induced from $\alg$, we get
 \bd
\Hom_G\big(\chi^{aN}, H^0\big(\check{C}^\bullet\big(\on{gr}\; M_c(\chi^{bN})\big)\big)\big)
\cong 
\Hom_G\big(\chi^{aN}, H^0\big(\check{C}^\bullet(\mathscr{L}^{bN})\big)\big)
\cong
H^0(\resol, \mathscr{L}^{bN-aN}).
\ed
The latter space is nonzero for $b-a\geq 0$ by very ampleness of $\mathscr{L}$.  Now 
Corollary \ref{gr commutes with H^0}  implies that 
$\on{gr}\; \Hom_G\big(\chi^{aN}, \check{C}^\bullet(M_c(\chi^{bN}))\big) \neq 0$, 
 thereby establishing (b).

If $M$ is any object of $(\alg,G,c)-\on{mod}$, then by Lemma \ref{computing twisted} we have
\begin{equation}\label{sectionsAGc}
\mathbb{R}\Hom\big(\E_{\resol}(c)\otimes\chi^{aN}, \pi_c M\big) 
= \Hom_G(\chi^{aN}, \mathbb{R}\Gamma_c\big(\pi_c(M)\big). 
\end{equation}
For any complex $x$ in $\cE_{\resol}(c)-\on{mod}$, let $\wt{x} = \mathbb{R}\Gamma_c(x)$ (a complex in $(\alg,G,c)-\on{mod}$); then 
$\pi_c(\wt{x}) = x$.  
Now Theorem \ref{Cech calculates} shows that $\mathbb{R}\Gamma_c\big(\pi_c(\wt{x})\big)\simeq \check{C}^\bullet(\wt{x})$, and hence by reductivity of $G$ that 
\bd
\mathbb{R}\Hom\big(\E_{\resol}(c)\otimes\chi^{aN}, x\big) \simeq \Hom_G\big(\chi^{aN},\check{C}^\bullet(\wt{x})\big).
\ed
Now, assume $x\not\simeq 0$; then there exists an $i$ such that, writing $x$ as $\dots \rightarrow x^{i-1} \xrightarrow{d^{i-1}} x^i \xrightarrow{d^i} x^{i+1}\rightarrow\dots$, we have $\mathcal{H}^i(x) = \ker(d^i)/\on{im}(d^{i-1}) \neq 0$.  By exactness of $\pi_c$, it follows that $\mathcal{H}^i(\wt{x}) \neq 0$.  Choose a finitely generated subobject $M\subseteq \mathcal{H}^i(\wt{x})$ in $(\alg,G,c)-\on{mod}$ such that 
$0\neq \pi_c(M) \subseteq \pi_c\big(\mathcal{H}^i(\wt{x})\big) = \mathcal{H}^i(x)$.  Let $N$ be the preimage of $M$ in $\ker(\wt{d}^i) \subseteq \wt{x}^i$, and choose a finitely generated $N'\subseteq N$ such that $N'$ surjects onto $M$ under 
$\ker(\wt{d}^i)\twoheadrightarrow \mathcal{H}^i(\wt{x})$.  
\begin{claim}
For $a\gg 0$, the natural map 
\bd
\Hom\big(\E_{\resol}(c)\otimes\chi^{-aN}, \pi_c N'\big)\longrightarrow
\Hom\big(\E_{\resol}(c)\otimes\chi^{-aN}, \pi_c M\big)
\ed
is nonzero.
\end{claim}
If the claim is true, there are maps 
\bd
\E_{\resol}(c)\otimes\chi^{-aN}[-i]\rightarrow \pi_c N'[-i]\rightarrow x
\ed
and the composite induces a nonzero map on $\mathcal{H}^i$, namely $\E_{\resol}(c)\otimes\chi^{-aN}\rightarrow \pi_cM\subseteq \mathcal{H}^i(x)$. 
Hence, if the claim is true then statement (c) above is true, and the proposition is proven.  
\begin{proof}[Proof of Claim.]
Observe that the desired map of Homs is computed by 
\begin{equation}\label{Cech N'}
\Hom_G\big(\chi^{aN}, H^0(\check{C}^\bullet(N'))\big) \rightarrow
\Hom_G\big(\chi^{aN}, H^0(\check{C}^\bullet(M)\big).
\end{equation}
Equip $N'$ with a good filtration as an object of  $(\alg,G,c)-\on{mod}$, and give $M$ the induced filtration as a quotient, so $\on{gr}(N')\rightarrow \on{gr}(M)$ is surjective.  Then 
\begin{equation}\label{surj}
\Hom_G\big(\chi^{aN}, H^0(\check{C}^\bullet(\on{gr}(N'))\big) \rightarrow
\Hom_G\big(\chi^{aN}, H^0(\check{C}^\bullet(\on{gr}(M))\big) \;
\text{is surjective for $a \ll 0$}
\end{equation}
 by ampleness of $\mathscr{L}$; fix such an $a$.  
  As in Corollary \ref{gr commutes with H^0}, for any $(\alg, G)$-module $P$ with good filtration,
 \bd
 \on{gr}\big(H^0(\check{C}^\bullet P)^{G}\big) \cong H^0((\check{C}^\bullet(\on{gr}\, P)\big)
\ed
functorially;
thus \eqref{surj} implies that \eqref{Cech N'} is surjective on the level of associated graded modules.  It follows from 
\eqref{gr M=0 implies M=0} that \eqref{Cech N'} is surjective; since the target module has nonzero associated graded for $a \gg 0$ by ampleness of $\mathscr{L}$, the claim follows.
\end{proof}
\noindent
As noted, this proves the proposition.
\end{proof}

\section{Direct and Inverse Image Functors and Serre Duality}
In this section we define the functors that we use to relate the microlocal category with the derived category of $U_c$-modules.  We also state a version of Serre Duality in the quantum setting that will allow us to establish 
derived equivalences later.   As in Section \ref{section 4}, we assume throughout that $G$ is connected reductive.
\subsection{Direct and Inverse Image Functors}
Let $f: \resol\rightarrow X$ denote the morphism from the GIT quotient to the affine quotient.  We define a quantization of the functor $f^*$ as follows.  
We define
\bd
\on{Loc}: U_c-\on{mod}\rightarrow (\alg,G, c)-\on{mod}
\ed
by 
$\on{Loc}(N) = M_c\otimes_{U_c}N$
(by Lemma \ref{twisted equivariant modules}, the functor $\on{Loc}$ takes values in $(\alg,G,c)-\on{mod}$).  
Recall that $\pi_c$ denotes the quotient functor $(\alg,G,c)-\on{mod} \rightarrow \E_{\resol}(c)-\on{mod}$. 
We then let $\map^* = \pi_c\circ \on{Loc}: U_c-\on{mod}\rightarrow \cE_{\resol}(c)-\on{mod}$.  
 
\begin{lemma}\label{lower *}
\mbox{}
\begin{enumerate}
\item We have $\mathbb{L}\map^* = \pi_c\circ\mathbb{L}\on{Loc}$.
\item The composite $\map_* \overset{\on{def}}{=} \mathbb{H}_c\circ \Gamma_c$ is right adjoint to $\map^*$, and $\mathbb{R}\map_* \overset{\on{def}}{=} \mathbb{H}_c\circ \mathbb{R}\Gamma_c$ is right adjoint to $\mathbb{L}\map^*$.
\item $\mathbb{R}\map_* = \mathbb{R}\Hom_{\E_{\resol}(c)-\on{mod}}(\E_{\resol}(c), -)$.
\item The functor $\mathbb{R}\map_*$ preserves colimits.
\end{enumerate}
\end{lemma}
We remark that a functor of abelian categories whose domain category is a Grothendieck category is known to admit right derived functors between unbounded derived categories; and a functor of abelian categories whose domain category has enough projectives is known to admit left derived functors between unbounded derived categories.  Hence $\mathbb{L}\map^*$ and $\mathbb{R}\map_*$ are defined.  
\begin{proof}[Proof of Lemma \ref{lower *}]
(1) follows from exactness of $\pi_c$; (2) is a formal consequence of (1).  
For (3), we observe that
\bd
\mathbb{R}\Hom_{\E_{\resol}(c)}(\E_{\resol}(c), N) = \mathbb{R}\Hom_{\E_{\resol}(c)}(\pi_c(M_c), N)
=  \mathbb{R}\Hom_{(\alg, G, c)}(M_c, \mathbb{R}\Gamma_c(N))
\ed
and apply statement (2) and the definition of $\mathbb{H}_c$.
As in Proposition \ref{compact generation}, $\pi_c M_c$ is compact; (4) then follows from (3).  
\end{proof}

\begin{lemma}\label{direct image of D}
We have $\mathbb{R}\map_*\E_{\resol}(c) = U_c$.
\end{lemma}
\begin{proof}
By adjunction, there is a natural map $M_c\rightarrow \mathbb{R}\Gamma_c\circ\pi_c M_c$.  Applying $\mathbb{H}_c$ gives a natural map $U_c\rightarrow \mathbb{R}\map_*\cE_{\resol}(c)$.  
Theorem \ref{Cech calculates} identifies this map with the natural map $U_c\rightarrow \check{C}^\bullet(M_c)^G$.  
By Proposition \ref{Rees Cech} and Assumption \ref{assumptions}(3) the map becomes an isomorphism
on passing to associated graded modules, yielding the lemma.
\end{proof}

\begin{prop}\label{existence of right adjoint}
The functor $\mathbb{R}\map_*: D(\E_{\resol}(c))\rightarrow D(U_c)$ has a right adjoint $\map^!$.
\end{prop}
\begin{proof}
By Lemma \ref{lower *}(4), Proposition \ref{compact generation}(2), \cite[Theorem 4.1]{Neeman} implies this.
\end{proof}

\begin{prop}
\label{finiteness of direct images}
If $M$ is a $(G,c)$-equivariant $\alg$-module that is finitely generated over $\alg$, then $\mathbb{R}\map_*\pi_c(M)$ has cohomologies that are finitely generated over $U_c$.  Moreover, choosing a good filtration of $M$, $\mathbb{R}\map_*\pi_c\big(\cR(M)\big)$ has cohomologies that are finitely 
generated over $\cR(U_c)$.  
\end{prop}
\begin{proof}
Choose a good filtration of $M$.  
By Theorem \ref{Cech calculates}(2), Lemma \ref{lower *}(2) and Lemma \ref{basics of H}, it suffices to show that the cohomologies of $\check{C}^\bullet(M)^G$ and $\check{C}^\bullet\big(\cR(M)\big)^G$ are finitely generated over $U_c$, respectively $\cR(U_c)$.   Note that the cohomologies of 
$\check{C}^\bullet(M)^G$ are naturally filtered by the identification $\check{C}^\bullet(M)^G \cong \C[t]/(t-1)\otimes \check{C}^\bullet\big(\cR(M)\big)^G$.
By Theorem 5.7 in Chapter II of \cite{LvO}, we can prove that those cohomologies are finitely generated by proving that they have separated, exhaustive filtrations and that the associated graded modules are finitely generated over $\C[\mu\inv(0)]^G$.  The cohomologies are subquotients of direct sums of $Q_{S_I}^\mu(\alg)\otimes M$.  By Lemma 3.7 and Proposition 3.9 of \cite{AVV}, the induced filtrations on each $Q_{S_I}^\mu(\alg)\otimes M$ are separated and exhaustive.  Hence each cohomology inherits a separated and exhaustive filtration as a subquotient; and it suffices to prove  that the cohomologies of $\check{C}^\bullet(\on{gr}(M))^G$ are finitely generated over $\C[\mu\inv(0)]^G$.

Let $F$ denote the chosen $G$-stable good filtration of $M$.  Since $M$ is $(G,c)$-equivariant, its singular support lies in $\mu\inv(0)$.  Thus
$\on{gr}_F(M)$ is a $G$-equivariant finitely generated $\C[\mu\inv(0)]$-module.  Let $p: \mu\inv(0)^{ss}\rightarrow \resol$ denote the GIT quotient map; then
$\mathsf{M} = p^G_*\big(M|_{\mu\inv(0)^{ss}}\big)$ is a coherent $\theo_{\resol}$-module.  Since the complex
$\check{C}^\bullet(\on{gr}(M))^G$ represents $\mathbb{R}f_*\on{gr}(\mathsf{M})$ and $f$ is a projective morphism, the complex has coherent cohomologies.  
This proves the final statement of the previous paragraph, hence the proposition.
\end{proof}

Given a ring or sheaf of rings $R$, we write $\on{sPerf}(R)$ for the {\em strictly perfect derived category} of $R$, i.e. the full triangulated (or dg if one wants to work with enhancements) subcategory of the bounded derived category consisting of objects quasi-isomorphic to bounded complexes of {\em free} $R$-modules.   We write $\on{Perf}(R)$ for the {\em perfect derived category} of $R$, the smallest full (triangulated or dg) subcategory of the 
bounded derived category containing $R$ and closed under retracts.
One has
the following easy equivalences:
\begin{prop}
The functors $\mathbb{L}\map^*$, $\mathbb{R}\map_*$ induce adjoint pairs of functors 
\bd
\mathbb{L}\map^*: \on{sPerf}(U_c) \leftrightarrows \on{sPerf}(\E_{\resol}(c)): \mathbb{R}\map_*,
\ed
\bd
\mathbb{L}\map^*: \on{Perf}(U_c) \leftrightarrows \on{Perf}(\E_{\resol}(c)): \mathbb{R}\map_*.
\ed
Moreover, these functors define quasi-inverse equivalences between the two (strictly) perfect derived categories.  
\end{prop}
\begin{proof}
By Lemma \ref{direct image of D}, we have $\mathbb{R}\map_*\E_{\resol}(c) = U_c$.  By construction, 
$\mathbb{L}\map^*U_c = \E_{\resol}(c)$.  It is immediate that the derived functors $\mathbb{L}\map^*, \mathbb{R}\map_*$ induce functors of the perfect derived categories.  It then suffices to observe that the two functors induce bijections on the Homs between free modules, which is a consequence of Lemmas \ref{lower *}(3) and \ref{direct image of D}.  This proves the strictly perfect statement.  Since the two functors preserve direct sums, the statements for the perfect derived categories follow.
\end{proof}

\subsection{Noncommutative Duality}\label{NC duality section}
We will need the following version of duality.  
\begin{thm}\label{serre duality}
The canonical homomorphism
\begin{equation}\label{SD homom}
\mathbb{R}\Hom_{\E_{\resol}(c)}(M, \E_{\resol}(c)) \rightarrow \mathbb{R}\Hom_{U_c}(\mathbb{R}\map_*M, \mathbb{R}\map_*\E_{\resol}(c))
= \mathbb{R}\Hom_{U_c}(\mathbb{R}\map_*M, U_c)
\end{equation}
is an isomorphism for all $M\in D(\E_{\resol}(c))$.
\end{thm}
Note that by adjunction, the statement is immediately equivalent to:
\begin{corollary}\label{computation of f^!}
We have $\map^! U_c = \E_{\resol}(c)$.
\end{corollary}

\subsection{Proof of Theorem \ref{serre duality}}\label{SD proof}
The homomorphism in the theorem comes from an adjunction morphism 
\bd
\E_{\resol}(c)\rightarrow \map^! \mathbb{R}\map_*\E_{\resol}(c) = \map^! U_c.
\ed
The statement of the theorem is equivalent to this adjunction morphism being an isomorphism in $D(\E_{\resol}(c))$.  It suffices to check that the induced map
\bd
\Hom_{\E_{\resol}(c)}(-, \E_{\resol}(c))\rightarrow \Hom_{\E_{\resol}(c)}(-, \map^! U_c)
\ed
is an isomorphism for a set of generators of $D(\E_{\resol}(c))$: if it is, then the cone of the adjunction must be zero.  Thus, we may prove the theorem by checking
that \eqref{SD homom} is an isomorphism for a collection of compact generators of $D(\E_{\resol}(c))$.  By the proof of Proposition 
\ref{indecomposable}, the objects $(\E_{\resol}(c)\otimes\chi^\ell)[n]$, $\ell\in\mathbb{Z}$, generate, and thus in particular it is sufficient to prove that \eqref{SD homom} is an isomorphism for objects $\pi_c(M_c(\chi^\ell))[n]$.  Via the adjunction $(\mathbb{L}\map^*, \mathbb{R}\map_*)$, moreover, it is enough to show that the natural map
\begin{equation}\label{the one to show}
\mathbb{R}\Hom_{\E_{\resol}(c)}\big(\pi_c(M_c(\chi^\ell)), \E_{\resol}(c)\big) \rightarrow \mathbb{R}\Hom_{\cE_{\resol}(c)}\big(\mathbb{L}\map^*\mathbb{R}\map_*\pi_c(M_c(\chi^\ell)), \E_{\resol}(c)\big)
\end{equation}
induced by the adjunction $\mathbb{L}\map^*\mathbb{R}\map_* \pi_c(M_c(\chi^\ell))\rightarrow \pi_c(M_c(\chi^\ell))$ is an isomorphism for all $\ell$. 
We will establish this using the corresponding map on Rees modules.

Write $M = M_c(\chi^\ell)$ and equip it with the standard filtration and let $\cR(M)$ be the corresponding Rees module.  By Proposition \ref{finiteness of direct images}, $\check{C}^\bullet\big(\cR(M)\big)^G$ is a bounded complex with cohomologies that are finitely generated over $\cR(U_c)$.  By Lemma \ref{replacement lemma}, we may find a bounded-above complex $F^\bullet$ of 
finitely generated free $\cR(U_c)$-modules and a quasi-isomorphism $F^\bullet \rightarrow \check{C}^\bullet\big(\cR(M)\big)^G$ that is compatible with 
$\C[t]$-base change.  We obtain a diagram $\cR(M)\rightarrow \check{C}^\bullet\big(\cR(M)\big) \leftarrow \cR(M_c)\otimes_{\cR(U_c)} F^\bullet$ 
of $(G,c)$-equivariant graded $\cR(\alg)$-modules.  Applying $\pi_c$ and using Theorem \ref{Cech calculates}(1), we get a map
$\pi_c\cR(M) \leftarrow \pi_c \cR(M_c)\otimes_{\cR(U_c)} F^\bullet$.  This induces a map 
\bd
\mathbb{R}\Hom\big(\pi_c \cR(M), \pi_c\cR(M_c)\big) \rightarrow \mathbb{R}\Hom\big(\pi_c \cR(M_c)\otimes_{\cR(U_c)} F^\bullet, \pi_c\cR(M_c)\big).
\ed
Applying the adjunction between $\pi_c$ and $\mathbb{R}\Gamma_c$, we obtain a map
\begin{equation}\label{another label}
\mathbb{R}\Hom\big(\cR(M), \mathbb{R}\Gamma_c\pi_c\cR(M_c)\big) \rightarrow \mathbb{R}\Hom\big(\cR(M_c)\otimes_{\cR(U_c)} F^\bullet, \mathbb{R}\Gamma_c\pi_c\cR(M_c)\big).
\end{equation}
Since both $\cR(M)$ and $\cR(M_c)\otimes_{\cR(U_c)} F^\bullet$ are complexes of projectives in $(\cR,G,c)-\on{mod}$, using Theorem \ref{Cech calculates}(2), \eqref{another label} is identified with a map
\begin{equation}\label{Rees Hom to use}
\Hom_{(\cR,G,c)}\big(M, \check{C}^\bullet(\cR(M_c))\big) \rightarrow \Hom_{(\cR,G,c)}\big(\cR(M_c)\otimes_{\cR(U_c)} F^\bullet, \check{C}^\bullet(\cR(M_c))\big).
\end{equation}

We want to apply $\C[t]/(t)\otimes -$ to \eqref{Rees Hom to use}.  Since $M$ and $\cR(M_c)\otimes_{\cR(U_c)} F^\bullet$ are complexes of finitely generated modules that are $\C[t]$-flat, Lemma \ref{base change of Hom} implies that this amounts to computing the Hom of base-changed modules.  By construction, 
$\C[t]/(t)\otimes \cR(M_c)\otimes_{\cR(U_c)} F^\bullet \simeq \C[\mu\inv(0)]\otimes_{\C[X]}\mathbb{R}f_*\theo_{\resol}$.  Hence the above map reduces mod $t$ to
\bd
\Hom_{(\C[\bigvar],G)}\big(\on{gr}\, M, \check{C}^\bullet(\mu\inv(0)^{ss})\big) \rightarrow 
\Hom_{(\C[\bigvar],G)}\big(\C[\mu\inv(0)]\otimes_{\C[X]}\mathbb{R}f_*\theo_{\resol}, \check{C}^\bullet(\mu\inv(0)^{ss})\big).
\ed
Using the adjunction between direct and inverse image along the inclusion $\mu\inv(0)^{ss}\rightarrow \mu\inv(0)$, this is identified with 
$\Hom_{\resol}(\mathscr{L}^\ell, \theo_{\resol})\rightarrow \Hom_{\resol}(\mathbb{L}f^*\mathbb{R}f_*\mathscr{L}^\ell, \theo_{\resol}).$
This is an isomorphism by ``classical'' Serre-Grothendieck duality (using that $f:\resol\rightarrow X$ has trivial relative dualizing sheaf).

We next want to conclude that \eqref{Rees Hom to use} is an isomorphism.  Note that both complexes appearing there are $\C[t]$-torsion free, hence
$\C[t]$-flat, and have finitely generated (as $\cR(U_c)$-modules) cohomologies.  Hence (Lemma \ref{replacement lemma}) they can be replaced by a map $\phi$ of bounded above complexes of finitely generated free modules $\cR(U_c)$-modules, compatibly with $\C[t]$-base change.  The cone on 
$\phi$ is then bounded above and consists of finitely generated free $\cR(U_c)$-modules.  Tensoring the cone with $\C[t]/(t)$ yields a complex with all cohomologies vanishing by the conclusion of the previous paragraph; hence by Proposition \ref{filtrations and ass gr}, $\on{Cone}(\phi)\simeq 0$, i.e., 
\eqref{Rees Hom to use} is an isomorphism.  

Finally,
since $M$ and $\cR(M_c)\otimes_{\cR(U_c)} F^\bullet$ are complexes of finitely generated modules that are $\C[t]$-flat, tensoring \eqref{Rees Hom to use} with $\C[t]/(t-1)$ gives (again by Lemma \ref{base change of Hom}) the isomorphism
\bd
\mathbb{R}\Hom\big(M, \mathbb{R}\Gamma_c\pi_c M_c\big) \rightarrow \mathbb{R}\Hom\big(M_c\otimes_{U_c} (\C[t]/(t-1)\otimes F^\bullet), \mathbb{R}\Gamma_c\pi_c M_c\big).
\ed
Using the adjunction $(\pi_c, \mathbb{R}\Gamma_c)$ and identifying $\pi_c (M_c\otimes_{U_c}\C[t]/(t-1)\otimes F^\bullet)$ with $\mathbb{L}\map^*\mathbb{R}\map_*M$,
we thus obtain that \eqref{the one to show} is an isomorphism, as desired.\hfill\qedsymbol

\section{Derived Equivalence}

\subsection{Equivalence Theorem}
Let $F: C\rightarrow D$ be an exact functor of unbounded derived categories (of abelian categories that contain arbitrary direct sums).  We say $F$ is {\em bounded} if it induces a functor from the bounded derived subcategory to the bounded derived subcategory: that is, for every $a\leq b$ and every object $c\in C^{[a,b]}$, there is some $N$ such that 
$F(c)\in D^{[a-N, b+N]}$.  We say $F$
 is {\em bounded by $N$} if, for all $a \leq b$,
$F(C^{[a,b]}) \subseteq D^{[a-N,b+N]}$.  We will say $F$ is {\em uniformly bounded} or {\em has finite Tor-dimension} if it is bounded by $N$ for some $N$.  

\begin{lemma}\label{bounded and uniformly bounded}
Suppose $F:C\rightarrow D$ is an exact and right $t$-exact functor of unbounded derived categories.  Suppose, in addition, that $F$ is continuous (that is, it is a left adjoint).  Then:
\begin{enumerate}
\item $F$ is bounded if and only if $F$ is uniformly bounded. 
\item If $C=D(A)$ is the unbounded derived category of a locally noetherian abelian category $A$, then $F$ is bounded by $N$ if and only if $F(c)\in D^{[a-N, b+N]}$ for every $a\leq b$ and object $c\in C^{[a,b]}$ with noetherian cohomologies.
\end{enumerate}
\end{lemma}
\begin{proof}
Write $C=  D(A)$ as in part (2).  Note that $F$ is uniformly bounded if and only if there is some $N$ such that $F(c)\in D^{[-N, N]}$ for all $c\in A$. 

(1) Only one direction requires proof.  If $F$ is not uniformly bounded, find a sequence $c_i\in A$, $i\geq 1$,  such that $\tau_{\leq -i}\big(F(c_i)\big)\not\simeq 0$ (that is, $F(c_i)$ has a nonzero cohomology in some degree $\leq -i$).  Since $F$ is continuous (i.e. commutes with colimits), we find that $F(\oplus_i c_i) \simeq \oplus_i F(c_i)$ has unbounded cohomologies.  Thus $F$ is not bounded.

(2) It suffices to check the uniform bound on objects of $A$.  Since $F$ commutes with colimits and, by hypothesis, every object of $A$ is a colimit of noetherian objects, the conclusion follows.
\end{proof}

Given $F\in D(U_c)$, there is an adjunction morphism 
$F\xrightarrow{a}\mathbb{R}\map_*\mathbb{L}\map^*F$.
\begin{lemma}\label{faithful}
The adjunction 
\bd
F\xrightarrow{a}\mathbb{R}\map_*\mathbb{L}\map^*F
\ed
 is an isomorphism for all $F$ in $D(U_c)$.  
 \end{lemma}
 \begin{proof}
  Suppose that $F$ is perfect (equivalently, compact), and 
write $F\simeq P^\bullet$ where $P^\bullet$ is a bounded-above complex of free modules.  Then 
\bd
\mathbb{R}\map_*\mathbb{L}\map^*F \simeq \mathbb{R}\map_*\E_{\resol}(c)\otimes_{U_c} P^\bullet \simeq P^\bullet.
\ed
Furthermore, both functors $\mathbf{1}$ and $\mathbb{R}\map_*\mathbb{L}\map^*$ commute with colimits.  Since $D(U_c)$ is compactly generated, every object is a colimit of compact objects (Lemma \ref{ind-category}), and the lemma follows.
 \end{proof}
\noindent
 Note that Lemma \ref{faithful} can be seen as a special case of the projection formula.

The main result of the paper is the following.

\begin{thm}\label{main equiv thm}
Suppose that $\mathbb{L}\map^*$ is bounded.  Then:
\begin{enumerate}
\item The functor $\map^!$ preserves coproducts.
\item The functors $\mathbb{L}\map^*$ and $\mathbb{R}\map_*$ form mutually quasi-inverse equivalences of the unbounded derived categories,
\bd
\mathbb{L}\map^*: D(U_c)\leftrightarrows D(\E_{\resol}(c)): \mathbb{R}\map_*,
\ed
that restrict to equivalences of the bounded derived categories.
\end{enumerate}
\end{thm}
 Part (1) of the theorem is analogous to part of Theorem 1.2 of \cite{LN} (in the commutative case).
 
 \subsection{Proof of Theorem \ref{main equiv thm} and Corollaries}
The proof of Theorem \ref{main equiv thm} uses both left and right $U_c$-modules and $\alg$-modules; hence we begin by introducing some notation we will use in the proof when dealing with right $\alg$-modules or $U_c$-modules.  

Everything in previous sections goes through {\em mutatis mutandis} for for right $\alg$-modules.  
It will be convenient to write the functors $\mathbb{R}\map_*$ and $\mathbb{L}\map^*$ specifically for right $\alg$-modules or $U_c$-modules, for which we use the notation $\mathbb{R}{^{r}\map}_*$ and $\mathbb{L}^{r}\map^*$.  Thus, 
\bd
\mathbb{R}^{r}\map_* (M) = \Hom_{\on{mod}-\cE_{\resol}}(\pi_c M_c^\dagger, M) = \mathbb{R}\Hom_{\on{mod}-\alg}(M_c^\dagger, \mathbb{R}\Gamma_c\pi_c M)^G  \;\; \text{and} \;\;
\mathbb{L}^{r}\map^*(F) = \pi_c (F\otimes_{U_c} M_c^\dagger).
\ed
We write ${}^{r}\map^!$ for the right adjoint of $\mathbb{R}^{r}\map_*$.

Assume first that $\mathbb{L}^{r}\map^*$ is bounded.  We will then prove that $\map^!$  preserves coproducts and use this to conclude that $\mathbb{L}\map^*, \mathbb{R}\map_*$ form mutually quasi-inverse equivalences.  We will subsequently deduce that $\mathbb{L}^{r}\map^*$ and $\mathbb{R}^{r}\map_*$ also form mutually quasi-inverse equivalences (from which it is immediate that ${}^{r}\map^!$ preserves coproducts).   Repeating the arguments {\em mutatis mutandis} with the roles of left and right modules reversed, it is then clear that it is equivalent to assume either that $\mathbb{L}^{r}\map^*$ or $\mathbb{L}\map^*$ is bounded: either assumption leads to equivalence for both left and right modules.  Thus,  we will begin the proof by assuming that $\mathbb{L}^{r}\map^*$ is bounded.

  To prove that $\map^!$ preserves coproducts, it follows from Theorem 5.1 of \cite{Neeman} that it is enough to show that $\mathbb{R}\map_*$ takes a set of compact generators to compact objects. 
We will use the compact objects $M = \E_{\resol}(c)\otimes\chi^\ell = \pi_c(M_c(\chi^\ell))$ from Proposition \ref{indecomposable} and show that, for each such $M$, $\mathbb{R}\map_*M$ is
perfect; since these objects and their shifts generate, this suffices. 

\begin{lemma}\label{perfection and tor}
Suppose $\cM\in D(U_c)$ is a complex with only finitely many nonzero cohomologies and that for every $i$, $h^i(\cM)$ is a finitely generated $U_c$-module.  Then $\cM$ is perfect if and only if there is some $N$ such that, for every right $U_c$-module $F$,  $h^{N-i}\big(F\otimes^{\mathbb{L}}_{U_c} \cM\big) = 0$ for all $i\geq 0$.
\end{lemma}
\begin{proof}[Proof of Lemma.]
Only the ``if'' direction requires proof.  Fix an $N$ as in the statement of the lemma.
By Lemma \ref{replacement lemma}, $\cM$ is isomorphic to a bounded-above complex of finitely generated free $U_c$-modules, 
\bd
\cM \simeq P^\bullet =  [\dots \rightarrow P_{N-1} \xrightarrow{d^{N-1}} P_N \xrightarrow{d^N} P_{N+1}\rightarrow \dots \rightarrow P_k].
\ed
By the vanishing hypothesis for $F=U_c$, we conclude that $P^\bullet$ has no cohomologies in degrees less than $N+1$.
It follows that $P^\bullet \simeq \tau_{\geq N+1}P^\bullet$, where
\bd
\tau_{\geq N+1}P^\bullet  = [ \dots \rightarrow 0\rightarrow \on{Im}(d^N) \rightarrow P_{N+1}\rightarrow P_{N+2}\rightarrow \dots].
\ed
As usual, we also let
\bd
\sigma_{\geq N+1}P^\bullet = [\dots\rightarrow 0\rightarrow  P_{N+1}\rightarrow P_{N+2}\rightarrow \dots].
\ed
 We then get an exact triple
\bd
\sigma_{\geq N+1} P^\bullet   \longrightarrow \tau_{\geq N+1}P^\bullet \longrightarrow \on{Im}(d^{N})[-N]  \xrightarrow{[1]}.
\ed
Now take derived tensor products with a right $U_c$-module $F$; we get a long exact sequence
\bd
\dots \rightarrow 
h^{N-i}\big(F\otimes \tau_{\geq N+1} P^\bullet\big)
\rightarrow h^{N-i}\big(F\otimes \on{Im}(d^{N})[-N]\big) \rightarrow
h^{N-i+1}\big(F\otimes \sigma_{\geq N+1} P^\bullet\big)
  \rightarrow \dots.
\ed
Since $\tau_{\geq N+1} P^\bullet\simeq P^\bullet\simeq \cM$, the left-hand term $h^{N-i}\big(F\otimes \tau_{\geq N+1} P^\bullet\big)$ vanishes for $i\geq 0$ by the hypothesis of the lemma.  On the other hand, since $P^\bullet$ is a complex of projectives, we have 
$h^{N-i+1}\big(F\otimes \sigma_{\geq N+1} P^\bullet\big) = 0$ for $N-i+1\leq N$, i.e., for $i\geq 1$.  Hence 
\bd
h^{-i}\big(F\otimes \on{Im}(d^{N})\big) =  h^{N-i}\big(F\otimes \on{Im}(d^{N})[-N]\big) = 0 \; \text{for $i\geq 1$}.
\ed
It follows from Lemma 4.1.6 of \cite{Weibel} that $\on{Im}(d^{N})$ is a flat $U_c$-module, and hence since $U_c$ is noetherian it is projective \cite[Theorem 3.2.7]{Weibel}.  Hence $\tau_{\geq N+1}P^\bullet$ is a strictly perfect complex quasi-isomorphic to $\cM$.  This completes the proof of the lemma.
\end{proof}

Returning to the proof of the theorem, by Proposition \ref{finiteness of direct images} and Lemma \ref{compact = perfect}, Lemma \ref{perfection and tor} shows that to prove that $\mathbb{R}\map_*M$ is compact, it suffices to prove that there exists an $N$ such that for each right $U_c$-module $F$, the complex
$F\overset{\mathbb{L}}{\otimes} \mathbb{R}\map_*M$ has no cohomologies in degrees less than $N$.  To prove the latter statement, we use the \v{C}ech complex of $M$ (and Proposition \ref{Cech of a quotient}) to write:
\begin{equation}\label{projection}
\begin{split}
F\overset{\mathbb{L}}{\otimes} \mathbb{R}\map_*M   
\cong & F\overset{\mathbb{L}}{\otimes}_{U_c} \big( M_c^\dagger\overset{\mathbb{L}}{\otimes}_\alg \check{C}^\bullet(M_c(\chi^\ell))\big)\\
\cong & \big(F\overset{\mathbb{L}}{\otimes}_{U_c} M_c^\dagger \big)\overset{\mathbb{L}}{\otimes}_{\alg} \check{C}^\bullet(\alg)\otimes_\alg (M_c(\chi^\ell)) \\
\cong &  \mathbb{R}\Gamma_c\big(\mathbb{L}^{r}\map^*(F)\big)\overset{\mathbb{L}}{\otimes}_{\alg} (M_c(\chi^\ell)).
\end{split}
\end{equation}
We note that we are entitled to omit taking $G$-invariants starting from the first line of \eqref{projection} by
Lemma \ref{basics of H} since $M_c(\chi^\ell)$ is $(G,c)$-equivariant.  
Since $M_c(\chi^\ell)$ has a finite projective $\alg$-resolution (for example by Theorem 2.3.7 of \cite{Bjbook}), the final term in \eqref{projection} has bounded cohomology 
provided $\mathbb{L}^{r}\map^*$ is a bounded functor.  Thus, assuming $\mathbb{L}^{r}\map^*$ is a bounded functor, the left-hand side of \eqref{projection} is also bounded.

Suppose next that $\map^!$ preserves coproducts.  
To prove that $\mathbb{L}\map^*$ and $\mathbb{R}\map_*$ are mutually quasi-inverse, we apply Proposition \ref{BMR lemma}.
Condition (C1) of the proposition is satisfied by Proposition \ref{compact generation}(3).
Condition (C2) of the proposition is satisfied by Lemma \ref{lower *}(4).  That condition (C3) of the proposition holds was established at the beginning of this proof.  
Condition (C4) of the proposition is satisfied by 
Corollary \ref{computation of f^!}.  
Condition (C5) of the proposition is satisfied by Lemma \ref{faithful}. 
Condition (C6) of the proposition is satisfied by Proposition \ref{indecomposable}.  It follows that $\mathbb{L}\map^*$ and $\mathbb{R}\map_*$ are mutually quasi-inverse.

Next we use the conclusion of the previous paragraph to prove that $\mathbb{L}^{r}\map^*$ and $\mathbb{R}^{r}\map_*$ are mutually quasi-inverse.  Since $\mathbb{L}\map^*$ is an equivalence, we have that $\map^! U_c = \mathbb{L}\map^* U_c = \pi_c M_c$.  Now, suppose $M= \pi_c \wt{M} \in D^b(\on{mod}-\cE_{\resol}(c))$ is nonzero and write 
\bd
M' = \mathbb{D}(M) = \pi_c \mathbb{R}\Hom_{\on{mod}-\alg}(\wt{M}, \alg) \in D^b(\cE_{\resol}(c)-\on{mod}).
\ed
  Then $M'\not\simeq 0$, and so 
\begin{align*}
\mathbb{R}\Hom(\pi_c M_c^\dagger, M)  & \cong \mathbb{R}\Hom_{\cE_{\resol}(c)}\big(M', \mathbb{D}(\pi_c M_c^\dagger)\big) 
  =  \mathbb{R}\Hom_{\cE_{\resol}(c)}\big(M', \pi_c M_c\big) \\
 & =  \mathbb{R}\Hom_{\cE_{\resol}(c)}\big(M', \map^! U_c\big) 
\cong \mathbb{R}\Hom_{U_c}\big(\mathbb{R}\map_*M', U_c\big)  \neq 0
\end{align*}
where the last inequality follows (via Lemma \ref{Gorenstein} and \cite[Proposition~1.3]{YZ}) from the fact that $\mathbb{R}\map_*$ has already been proven to be an equivalence.  It follows that the right orthocomplement of the image under $\mathbb{L}^{r}\map^*$ of $D^b(\on{mod}-U_c)$ in
$D^b(\on{mod}-\cE_{\resol}(c))$ is zero, which implies that $\mathbb{L}^{r}\map^*$ is an equivalence.
This completes
the proof of the theorem.\hfill\qedsymbol
\begin{remark}
In particular, it follows from the proof that $\mathbb{L}\map^*$ is a bounded functor of left $U_c$-modules if and only if it is a bounded functor of right $U_c$-modules, even in situations in which this may not imply finite global dimension (cf. Corollary \ref{first cor} below).
\end{remark}
\begin{corollary}\label{first cor}
Suppose that $\mu$, $X$, $\resol$, and $f: \resol\rightarrow X$ satisfy Assumptions
\ref{assumptions}.  
Consider the following assertions.
\begin{enumerate}
\item The algebra $U_c$ has finite global dimension.
\item The functor $\mathbb{L}\map^*$ is cohomologically bounded (that is, preserves complexes with cohomologies in only finitely many degrees).
\item The functor $\mathbb{L}\map^*$ induces an exact equivalence of bounded derived categories,
\bd 
\mathbb{L}\map^*: D^b(U_c)\rightarrow D^b(\cE_{\resol}(c)).
\ed
\end{enumerate}
Then (1)$\implies$(2)$\implies$(3).  If $G$ acts freely on $\mu\inv(0)^{ss}$, so $\resol$ is smooth, then (3)$\implies$(1).
\end{corollary}

\begin{proof}
(1) $\implies$ (2) is clear.  Theorem \ref{main equiv thm} gives (2) $\implies$ (3).  

We wish to prove that (3) $\implies$ (1).  
By Theorem 4.1.2 of \cite{Weibel}, to prove that $U_c$ has finite global dimension it suffices to show that there is an $N$ such that for all $L\in U_c-\on{mod}$,
the functor $\mathbb{R}\Hom_{U_c}(L, -)$ is bounded by $N$.  
Via $\mathbb{L}\map^*$ and Lemma \ref{bounded and uniformly bounded}, this is equivalent to showing $\mathbb{R}\Hom_{\cE_{\resol}(c)}(L, -)$ is uniformly bounded for every $L\in \cE_{\resol}(c)-\on{mod}$ with a bound independent of $L$.  

Writing $L=\pi_c(\wt{L})$ for some $\wt{L}\in (\alg,G,c)-\on{mod}$, this amounts to showing that $\mathbb{R}\Hom_{(\alg,G,c)}(\wt{L}, \mathbb{R}\Gamma(-))$ is uniformly bounded independent of $L$.    By Theorem \ref{Cech calculates} and a spectral sequence argument, this reduces to showing that for every finite set $S$ of semi-invariants as in Section \ref{Equivariant Microlocal Modules}, the functor
$\mathbb{R}\Hom_{(\D,G,c)}(\wt{L}, j_{S\ast}(-))$ with domain $Q_S(\alg)-\on{mod}$ is uniformly bounded independent of $L$.  By adjunction and exactness of $j_S^*$, this is equivalent to proving uniform boundedness of $\mathbb{R}\Hom(j_S^*\wt{L}, -)$.  Now the uniform boundedness follows from smoothness of $\resol$ by Proposition 3.1 of \cite{Bj}.
\end{proof}

\begin{corollary}\label{fin gl dim}
If $U_c$ has finite global dimension, then $\mathbb{L}\map^*$ is an equivalence of categories.
\end{corollary}

\subsection{Flag Varieties}\label{flag varieties}
There is a nice class of algebraic varieties called {\em Mori dream spaces} \cite{HK}, which includes flag varieties, smooth projective toric varieties, and (more generally) spherical varieties.  It is shown in \cite[Corollary~2.4]{HK} that such varieties arise as GIT quotients of (possibly singular) affine varieties by torus actions.  

We are interested in the case when the Mori dream space is nonsingular, particularly the examples of complete flag varieties $G/B$ (which are GIT quotients of the {\em basic affine space} $\overline{G/U}$ by the action of the torus $H = B/U$ at a dominant regular character $\chi$), and we focus on this last example.  We intend to return to the more general setting of Mori dream spaces elsewhere.

There is a standard embedding of $\overline{G/U}$ in the product $\mathbb{V} = V_1\times\dots\times V_r$ of fundamental representations of $G$ (here we assume $G$ is semisimple, connected, and simply connected).  Let $\pi: T^*\mathbb{V}\rightarrow \mathbb{V}$ denote the projection.  It is easy to compute that, for a dominant regular $\chi: H\rightarrow \Gm$, the unstable locus of the action of $H$ on $T^*\mathbb{V}$ is exactly 
\bd
\bigcup_i \pi\inv\big(V_1\times \dots\times V_{i-1}\times\{0\}\times V_{i+1}\times \dots \times V_r\big),
\ed
and its intersection with $\pi\inv(\overline{G/U})$ is exactly $\pi\inv(\overline{G/U}\smallsetminus G/U)$.  Thus, it is natural to view the category of $(H,c)$-equivariant crystals on $\overline{G/U}$ as a natural analog of $(\D(\base),H,c)-\on{mod}$ for a smooth variety 
$\base$, and to generalize our results to crystals to realize $\D_{G/B}(c)-\on{mod}$ as a microlocal category $\cE_{\resol}(c)-\on{mod}$.  

Although this can apparently be done, it is easier to simply apply Proposition \ref{BMR lemma} directly in this setting.  Indeed, our proofs above easily adapt (sometimes with far less difficulty) to prove the following.  We use the standard notation $\on{Loc}(M) = \D_{G/B}(\lambda)\otimes_{U_\lambda} M$ in place of $\mathbb{L}\map^*$ (and $\mathbb{R}\Gamma$ in place of $\mathbb{R}\map_*$).  
\begin{prop}
Choose a character $\lambda: \mathfrak{h}\rightarrow \C$ and let $D(\D_{G/B}(\lambda))$ denote the derived category of $\lambda$-twisted $\D$-modules on $G/B$.  Then:
\begin{enumerate}
\item $D(\D_{G/B}(\lambda))$ is compactly generated.
\item $\mathbb{R}\Gamma: D(\D_{G/B}(\lambda)) \rightarrow D(U_c)$ preserves coproducts.
\item $\mathbb{R}\Gamma$ takes a set of compact generators to compact objects.
\item The adjunction $1\rightarrow \mathbb{R}\Gamma\circ \mathbb{L}\on{Loc}$ is an isomorphism.  
\item $D(\D_{G/B}(\lambda))$ is indecomposable.
\end{enumerate}
\end{prop}
In order to establish a derived analog of Beilinson-Bernstein localization, then, all that remains is to prove the analog of 
Proposition \ref{BMR lemma}(C4) for the generator $U_\lambda$.  This can be done just as we do it above if (and only if) the derived localization functor
$\mathbb{L}\on{Loc}$ is cohomologically bounded; note that the arguments use more familiar technology since the \v{C}ech complex, for example, is the usual one.  Derived localization then follows from:
\begin{lemma}[\cite{BMR}, Lemma 3.3.4 and Corollary 3.3.5]
The functor $\on{Loc}  = \mathbb{L}\map^*$ is cohomologically bounded provided $\lambda$ is regular. 
\end{lemma}
\begin{corollary}[Derived Beilinson-Bernstein]
The categories $D(\D_{G/B}(\lambda))$ and $D(U_{\lambda})$ are equivalent if $\lambda$ is regular.
\end{corollary}

We remark that one---perhaps the only---advantage of this approach (besides the possible conceptual advantage that, unlike the ``classical'' proof, we do not use the left $B$-action on $G/B$ or $\overline{G/U}$, only, morally speaking, the $H$-action on 
$\overline{G/U}$) is that it seems to be relatively easy to prove that $\on{Loc}$ is a cohomologically bounded functor when $\lambda$ is a regular weight: cf. the elegant and short proof of \cite{BMR}.   By contrast, it seems to be no easier to prove that $U_\lambda$ has finite global dimension than to prove that localization holds.

\subsection{Wreath Product Spherical Rational Cherednik Algebras}\label{wreath}
Fix integers $\ell, n>1$ and let $\eta = e^{2\pi i/\ell}$.  

Let $\mu_\ell\cong \Gamma \subset SL(2)$ be the cyclic subgroup of $SL(2)$ generated by the matrix
$\sigma = \on{diag}(\eta, \eta\inv)$.  Let $\Gamma_n = S_n\wr \mu_\ell$ denote the semidirect product of $S_n$ and $\Gamma^n$; this is the {\em wreath product} of the cyclic group.  
For an element $\gamma \in \mu_\ell$ and an integer $1\leq i \leq b$, write $\gamma_i$ to mean 
$(1,\dots, 1, \gamma, 1, \dots, 1)\in \Gamma^n$, where $\gamma$ appears in the $i$th factor.  

Etingof-Ginzburg \cite{EG} associate to $\Gamma_n$ a {\em rational Cherednik algebra} defined as follows.
The elements
$s_{ij}\gamma_i\gamma_j\inv$, $1\leq i, j\leq n$, $\gamma\in \mu_\ell\smallsetminus\{1\}$ lie in a single conjugacy class $S$ of symplectic reflections in $S_n\wr \mu_\ell$.   Moreover, for $1\leq m\leq \ell-1$, 
the elements $(\sigma^m)_i$ are symplectic reflections lying in distinct conjugacy classes $C_\ell$ (the class does not depend on $i$).  It turns out
that this is a complete list of conjugacy classes of symplectic reflections in $S_n\wr\mu_\ell$.  Choosing 
$k\in\C$ and $c\in\C^{\ell-1}$, we get a function on the set of such conjugacy classes that assigns $k$ to 
$S$ and $c_m$ to $(\sigma^m)_i$.   Writing $V=(\C^2)^n$ with basis $x_i, y_i$, the Cherednik algebra
$\mathsf{H}_{k,c}$ is the quotient of $T^\bullet(V)\ast(S_n\wr\mu_\ell)$ by the relations:
\bd
[x_i, x_j] = 0, \hspace{2em} [y_i, y_j] = 0 \hspace{2em} \mbox{}  \text{for all $1\leq i, j\leq n$};
\ed
\begin{align*}
y_i x_i - x_iy_i = 1 + k\sum_{j\neq i} \sum_{\gamma\in\mu_\ell} s_{ij}\gamma_i\gamma_j\inv + 
\sum_{\gamma\in\mu_\ell \smallsetminus\{1\}} c_\gamma\gamma_i \hspace{2em} & \text{for all $1\leq i\leq n$};\\
y_ix_j - x_jy_i = -k\sum_{m=0}^{\ell-1}\eta^ms_{ij}(\sigma^m)_i(\sigma^m)_j\inv \hspace{2em} & \text{for $i\neq j$}.
\end{align*}
Writing $\ds e= \frac{1}{|S_n\wr\mu_\ell|} \sum_{w\in S_n\wr\mu_\ell} w$ for the symmetrizing idempotent, the {\em spherical subalgebra} is $\mathsf{U}_{k,c} \overset{\on{def}}{=} e\mathsf{H}_{k,c} e$.  

We recall, following \cite{O, Gordon}, the construction of the spherical subalgebra of the rational Cherednik
algebra via quantum Hamiltonian reduction.  Namely, let $Q$ be a cyclic quiver with $\ell$ vertices and 
cyclic orientation.  Choose an extending vertex denoted $0$.  Let $Q_\infty$ denote the quiver obtained by 
adding one vertex $\infty$ to $Q$ joined to $0$ by a single arrow.  Let $\overline{Q}_\infty$ denote the doubled quiver.  Let $\delta = (1,\dots, 1)$ be the affine dimension vector of $Q$ and let 
$\epsilon = e_\infty + n\delta$, a dimension vector for $Q_\infty$.  Let 
\bd
\on{Rep}(Q_\infty, \epsilon) = \left(\bigoplus_{r=0}^{\ell-1} \Hom(\C^n, \C^n)\right)\oplus \C^n
 = \{(X_0,\dots, X_{\ell-1}, i)\}.
 \ed
 Let $G = \prod_{r=0}^{\ell-1} GL(n)$ act on  $\on{Rep}(Q_\infty, \epsilon)$ by
 \bd
 (g_0, \dots, g_{\ell-1})\cdot (X_0,\dots, X_{\ell-1}, i)  = (g_0X_0g_1\inv, g_1X_1g_2\inv, \dots, g_{\ell-1}X_{\ell-1}g_0\inv, g_0i).
 \ed
 
 For each $(c_0,\dots, c_{\ell-1})\in\C^\ell$ and $k\in\C$, define a Lie algebra character $\chi_{k, c}: \mathfrak{g}\rightarrow \C$ as a sum $\chi_{k,c} = \chi_c+ \chi_k$ by:
 \begin{enumerate}
 \item
 $\chi_c(X_0,\dots, X_{\ell-1}) = \sum_{r=0}^{\ell-1}C_r\on{Tr}(X_r)$, where
 \bd
 C_r =\ell\inv\Big(1-\sum_{m=1}^{\ell-1} \eta^{mr} c_m\Big) \;\text{for $1\leq r\leq \ell-1$, and}\; 
 C_0 = \ell\inv\Big(1-\ell-\sum_{m=1}^{\ell-1}c_m\Big).
 \ed
 \item 
 $\chi_k(X_0,\dots, X_{\ell-1}) = k \on{Tr}(X_0)$.
 \end{enumerate}

Let $U_{k,c}$ denote the algebra associated by quantum Hamiltonian reduction to the $G$-action on 
$W = \on{Rep}(Q_\infty, \epsilon)$ with Lie algebra character $\chi_{k,c}$ (note that we are suppressing 
dependence on $n$).  Then: 
\begin{thm}[\cite{O, Gordon}]
$U_{-k,-c}$ is isomorphic to the spherical subalgebra $\mathsf{U}_{k,c}$
of the rational Cherednik algebra $\mathsf{H}_{k,c}$ associated to $\Gamma_n$ with parameters $k,c$. 
\end{thm} 
By Theorem 3.2.3 of \cite{GG}, the moment map for the $G$-action on $W = \on{Rep}(Q_\infty, \epsilon)$ 
is flat.  By Lemma \ref{sympresol}, Assumptions \ref{assumptions}(2a),(2b),(3) will hold whenever $\resol\rightarrow X$ is a symplectic resolution.  See \cite{Gordon2} for a discussion of the (nonempty) set of characters $\chi$ of $G$ for which $\mu\inv(0)/\!\!/_{\chi} G \cong \on{Hilb}^n(\wt{\C^2/\mu_\ell})$, the Hilbert scheme 
of $n$ points on the smooth surface $\wt{\C^2/\mu_\ell}$ obtained as the minimal resolution of the quotient
singularity $\C^2/\mu_\ell$.

Finally, we need to know for which values of $k,c$ the functor $\mathbb{L}\map^*$ is bounded. 

By Theorem 5.5 of \cite{E}, this fails if and only if $(k, C)$ is an {\em aspherical value} of the parameter.  The aspherical values are classified 
by \cite{DG}.  Changing notation to be consistent with what we have written above, one finds:
\begin{corollary}[Theorem~1.1 of \cite{DG} and Theorem \ref{main thm intro}]\label{cyclotomic global dimension}
The quantum Hamiltonian reduction algebra $U_{-k,-C}$ has infinite global dimension in the following cases.  
\begin{enumerate}
\item $k = \frac{u}{v}$ for integers $1\leq u<v\leq n$.
\item There exist a partition $\lambda\in P_n$, an integer
$0\leq a \leq \ell-1$, and an integer
$b\not\equiv 0 \mod \ell$ with $1\leq b \leq a + \ell(\on{ln}(\lambda) -1)$ for which 
\[
\frac{2b}{\ell} - ([\frac{b-a}{\ell}]+1) = \sum_{s=1}^b C_{a-b+s} -k(\lambda_1-\on{ln}(\lambda))
\]
\end{enumerate}
In all other cases, $U_{-k, -C}$ has finite global dimension, and hence derived microlocalization holds.
\end{corollary}

\subsection{Applications In Other Contexts}
As the structure of the proof of Theorem \ref{main equiv thm} suggests, the strategy used to establish derived equivalences in this paper is applicable in a number of other contexts.

\subsubsection{Non-Affine Quantizations}\label{good quotient}
Theorem \ref{main thm intro} generalizes to non-affine situations in which a good quotient (in the GIT sense) exists.  More precisely,
suppose $\bigvar$ is a smooth symplectic $G$-variety with moment map $\mu:\bigvar\rightarrow\mathfrak{g}^*$ and that $\mu\inv(0)$ comes with a good quotient map $q: \mu\inv(0)\rightarrow X$: in other words, an affine $G$-invariant morphism so that $q_*^G\theo_{\mu\inv(0)} = \theo_X$.  Suppose that $\bigvar$ comes equipped with a filtered quantization $\alg$, i.e. a sheaf of filtered rings whose associated graded is $\theo_{\bigvar}$, 
satisfying the conditions of Section \ref{quantum reduction}.  Then $U_c$ is naturally defined as a sheaf of rings on $X$, and the microlocal category
$D(\cE_{\resol}(c))$ quantizes the GIT quotient $\resol = \mu\inv(0)^{\chi\on{-ss}}/\!\!/G$.   An analog of Theorem 
\ref{main thm intro} follows in this setting by standard sheaf properties of DG categories, which in particular yields the following 
implication.
\begin{thm}\label{second thm intro}
Suppose that $q: \mu\inv(0)\rightarrow X$ is a good quotient (in the GIT sense).  
Suppose that $\mu$, $X$, $\resol$, and $f: \resol\rightarrow X$ satisfy Assumptions
\ref{assumptions}. If the sheaf of rings $U_c$ is locally of finite global dimension and $\resol$ is smooth, then 
the functor
\bd 
\mathbb{L}\map^*: D(U_c)\rightarrow D(\cE_{\resol}(c))
\ed
defines an exact equivalence of bounded derived categories.
\end{thm}
Corollary \ref{Ch alg of curve} provides an application. 

\subsubsection{Deformation Quantization Modules}
\label{DQ}
Another natural context to study is that of  an arbitrary $\mathbb G_m$-equivariant symplectic resolutions $f\colon \resol \to X$, where $X$ is a normal affine cone. Here, mimicking \cite{KR}, it is natural to take an ample sequence of line bundles $\mathscr L_i$, and quantizations thereof. For an appropriate category of ``good'' modules over the canonical quantization $\cW_{\resol}$ of $\resol$ one should obtain a derived equivalence between $D^b(\cW_{\resol}$-mod) and $D^b(U_\hbar$-mod) where $U_{\hbar} =\on{End}_{\cW_{\resol}}(\cW_{\resol})$. We plan to return to this in future work.

\section{Appendix: \v{C}ech Complexes}\label{proof of serre}
In this section, we lay out the basics of the \v{C}ech complex for microlocalizations of $\alg$.  The story is essentially standard (cf. \cite{VW, VW2} for the basics), but we do not know a reference that
does everything we need in the form we need.  As elsewhere in the paper, we assume that $G$ is a connected reductive group.

\subsection{Modules and Microlocalization}
Let $S\subset \C[\bigvar]$ be a subset of nonzero homogeneous elements (recall that $\C[\bigvar]$ is graded via its identification with $\on{gr}\,\alg$).  We let $\overline{S}$ denote the smallest multiplicatively closed subset $\overline{S}$ of $\C[\bigvar]$ containing $S$, and let $\overline{S}_{\on{sat}}$ denote any multiplicatively closed subset of $\alg$ whose collection of principal symbols is exactly $\overline{S}$ (later we will make a specific such choice).    
We also let $\wt{S}$ denote the subset consisting of elements of 
$\overline{S}_{\on{sat}}$ identified with 
homogeneous elements of the Rees algebra $\cR = \cR(\alg)$.

Let $\wt{Q}^\mu_{\wt{S}}(\alg)$ denote the microlocalization of $\cR(\alg)$ at $\wt{S}$ as defined in \cite{AVV}, and let 
$Q^\mu_S(\alg)$ denote the microlocalization of $\alg$.  More precisely, $\wt{Q}^\mu_{\wt{S}}(\alg)$ is a graded algebra over $\C[t]$ for which $\wt{Q}^\mu_{\wt{S}}(\alg)/(t-1)  = Q^\mu_S(\alg)$.
  We next  record a few basic properties of these algebras from \cite{AVV}:
\begin{lemma}\label{local basics}
\mbox{}
\begin{enumerate}
\item For every $a\neq 0$, $Q^\mu_S(\alg) \cong \wt{Q}^\mu_{\wt{S}}(\alg)/(t-a)$.  
\item The algebra $Q^\mu_S(\alg)$ comes
equipped with a homomorphism $\alg\rightarrow Q^\mu_S(\alg)$ that makes $Q^\mu_S(\alg)$ flat over $\alg$ on both sides.  Similarly, there is a natural 
graded homomorphism $\cR\rightarrow \wt{Q}^\mu_{\wt{S}}(\alg)$ making $\wt{Q}^\mu_{\wt{S}}(\alg)$ flat over $\cR$ on both sides. 
\item All elements $s\in \alg$ whose
symbol lies in $\overline{S}$ become invertible in $Q^\mu_S(\alg)$.  
\item  $Q^\mu_S(\alg)$ comes equipped with a 
filtration $F_\bullet$ for which $F_k(Q^\mu_S(\alg)) \cap \alg = \alg^k$.  Moreover, $\on{gr}_F(Q^\mu_S(\alg)) \cong \overline{S}^{-1}\C[\bigvar]\cong \wt{Q}^\mu_{\wt{S}}(\alg)/(t)$.
\item If $M$ is an $\alg$-module equipped with a good filtration, then $Q^\mu_S(\alg)\otimes_{\alg} M$ is naturally a filtered $Q^\mu_S(\alg)$-module with associated graded isomorphic to $\overline{S}\inv \on{gr}(M)$. 
\item If $M$ is a finitely generated graded $\cR(\alg)$-module, then $\wt{Q}^\mu_{\wt{S}}(M): = \wt{Q}^\mu_{\wt{S}}(\alg)\otimes_{\cR(\alg)} M$ is naturally a graded $\wt{Q}^\mu_{\wt{S}}(\alg)$-module with
$\wt{Q}^\mu_{\wt{S}}(M)/t\wt{Q}^\mu_{\wt{S}}(M) \cong \overline{S}\inv(M/tM)$.
\item If $M$ is a filtered $\alg$-module with Rees module $\wt{M}$, then for $a\neq 0$, $\wt{Q}^\mu_{\wt{S}}(\wt{M})/(t-a)\cong Q_S^\mu(\alg)\otimes_{\alg} M$. 
\end{enumerate}
\end{lemma}

Define $\on{Ind}_\alg^{Q_S^\mu(\alg)}(M)  = Q_S(\alg)\otimes_{\alg} M$.
One has an adjoint pair of functors:
\bd
\on{Ind}_\alg^{Q_S(\alg)} = Q_S(\alg)\otimes_{\alg} - : \alg-\on{mod} \leftrightarrows Q_S(\alg)-\on{mod}: \on{Res}_{\alg}^{Q_S(\alg)},
\ed
where $\on{Res}_{\alg}^{Q_S(\alg)}$ is the usual restriction functor.   One gets a similar adjoint pair for graded $\cR$-modules.

\begin{lemma}\label{kernel of microlocal induction}
We have
\bd
\on{Ker}\big(\on{Ind}_\alg^{Q_S(\alg)}\big) = \Big\{ M \; \Big|\; SS(M)  \subset \bigcup_{s\in S} V(s)\Big\},
\;\;
\on{Ker}\big(\on{Ind}_\cR^{\wt{S}^{-1}(\cR)}\big) = \Big\{ M \; \Big|\; SS(M)  \subset \bigcup_{s\in S} V(s)\Big\}.
\ed
\end{lemma}
\begin{proof}
The equalities of the lemma are immediate from Lemma \ref{local basics}(5) for finitely generated modules, and follow for all modules by writing arbitrary modules as colimits of finitely generated ones.
\end{proof}

\subsection{Equivariant Microlocal Modules}\label{Equivariant Microlocal Modules}
Suppose that the connected reductive complex group $G$ acts on $\bigvar$, that $\chi: G\rightarrow \Gm$ is a choice of character, and that $S$ is a set of homogeneous (for the grading on $\C[\bigvar]$ induced from $\on{gr}\,\alg$) $\chi^{\mathbb{N}}$-semi-invariants (i.e. consists of elements $s$ each of which is $\chi^{\ell(s)}$-semi-invariant for some $\ell(s)>0$).  Then $\overline{S}$ consists of $\chi^{\mathbb{N}}$-semi-invariants, together with $1$, and we  choose $\overline{S}_{\on{sat}}$ to consist of $1$ together with all $\chi^{\mathbb{N}}$-semi-invariants in $\alg$ whose principal symbols lie in $\overline{S}$; this is a multiplicatively closed subset of $\alg$.  
Note that then the $G$-action preserves the set of elements of $S$, $\overline{S}$, and $\overline{S}_{\on{sat}}$ up to scalars; as a result, $G$ acts naturally on
$Q_S^\mu(\alg)$ and $\wt{Q}^\mu_{\wt{S}}(\alg)$.

Recall that a vector $v$ in a $G$-representation $V$ is {\em $G$-rational} (or simply {\em rational} if the group is understood) if there is a finite-dimensional rational $G$-subrepresentation $W\subseteq V$ with $v\in W$.  Given a $G$-representation $V$, we let $V^{\on{rat}}$ denote the subspace of $G$-rational vectors: it is a $G$-subrepresentation of $V$. For the following Lemma it is convenient to note that a vector $v \in V$ is $G$-rational precisely if it lies in the image of a $G$-homomorphism $\rho\colon U \to V$ from a finite dimensional rational $G$-representation $U$.

\begin{lemma}\label{rational vectors of a free module}
If $V$ is a rational $G$-representation and $M$ is any $G$-representation, then $(M\otimes_{\C} V\big)^{\on{rat}} = M^{\on{rat}}\otimes_{\C} V$.
\end{lemma}
\begin{proof}
Since tensor products commute with colimits, we may assume that $V$ is finite-dimensional. If $u \in M\otimes V$ is a $G$-rational vector, pick a $G$-equivariant map $\rho\colon W \to M\otimes V$ from a finite-dimensional rational $G$-representation $W$ whose image contains $u$. We have a canonical morphism:
\bd
\Theta\colon \text{Hom}(W,M\otimes V) \to \text{Hom}(W\otimes V^*,M),
\ed
where if $f\colon W \to M\otimes V$, the linear map $\Theta(f)$ on $W\otimes V^*$ is given by the bilinear map $(w,\phi) \mapsto (1\otimes \phi)(f(w))$, ($w \in W, \phi \in V^*)$.  The map $\Theta$ is evidently $G$-equivariant and hence the map $\Theta(\rho)$ is also. Thus to see that $u$ lies in $M^\text{rat}\otimes V$ it suffices to note that $u$ lies in $\Theta(\rho)(W\otimes V^*)\otimes V$, which is clear for example by writing the map $\rho$ in terms of a basis for $V$ and $W$. But it is clear that $M^{\on{rat}}\otimes V \subseteq (M\otimes V)^{\on{rat}}$ and hence the Lemma follows.
\end{proof}

We write $\left(\wt{Q}^\mu_{\wt{S}}(\alg), G\right)-\on{Mod}$ for the category of weakly $G$-equivariant $\wt{Q}^\mu_{\wt{S}}(\alg)$-modules.

Write $Z(S)\subset \bigvar$ for the hypersurface defined as the union of hypersurfaces $V(s_i)$ for $s_i\in S$.  Note that $Z(S)$ is the complementary closed subset to the open immersion $\on{Spec}(\overline{S}\inv\C[\bigvar])\hookrightarrow \on{Spec}(\C[\bigvar])$.  By construction, $Z(S)$ is $G$-stable.  Given a Lie algebra character $c:\mathfrak{g}\rightarrow \C$, let $(\alg, G, c)-\on{mod}_{Z(S)}$ denote the localizing subcategory of  $(\alg, G, c)-\on{mod}$ whose objects are modules $M$ with $SS(M)\subseteq Z(S)$ (cf. Section \ref{microlocalization}).
\begin{prop}\label{right adj of pi_S}
\mbox{}
\begin{enumerate}
\item An object $M$ of $(\alg, G, c)-\on{mod}$ lies in $(\alg, G, c)-\on{mod}_{Z(S)}$ if and only if 
$\on{Ind}_\alg^{Q_S^\mu(\alg)}(M) = 0$.
\item The functor 
$\on{Ind}_\alg^{Q_S^\mu(\alg)}: (\alg, G, c)-\on{mod} \longrightarrow \big(\wt{Q}^\mu_{\wt{S}}(\alg), G\big)-\on{Mod}$
factors uniquely through the quotient category 
$(\alg, G, c)-\on{mod}/(\alg, G, c)-\on{mod}_{Z(S)}$, and the induced functor 
\bd
(\alg, G, c)-\on{mod}/(\alg, G, c)-\on{mod}_{Z(S)}\longrightarrow \big(\wt{Q}^\mu_{\wt{S}}(\alg), G\big)-\on{Mod}
\ed
 is faithful.
 \item The quotient functor 
$ \pi_S: (\alg, G, c)-\on{mod}\longrightarrow (\alg, G, c)-\on{mod}/(\alg, G, c)-\on{mod}_{Z(S)}$
 has right adjoint $\Gamma_S$ defined by $\Gamma_S\big(\pi_S(M)\big) = (Q_S^\mu(\alg)\otimes_{\alg} M)^{\on{rat}}.$  Moreover, the right 
 adjoint $\Gamma_S$ is exact. 
\end{enumerate}
Moreover, analogous statements hold for $(G,c)$-equivariant graded $\cR(\alg)$-modules.
\end{prop}
\begin{proof}
(1) follows from Lemma \ref{kernel of microlocal induction}.  (2) is immediate from the universal property of the quotient category.  

To prove (3), we begin by recalling that the right adjoint $\Gamma_S$ applied to $\pi_S(M)$ is constructed as the colimit over maps $M\rightarrow M'$ whose kernel and cokernel have singular support in $Z(S)$.  If $\tau(M)$ is the maximal submodule with singular support in $Z(S)$, then $\Gamma_S(\pi_S(M))= \Gamma_S(\pi_S(M/\tau(M)))$; in particular, we may assume that $\tau(M) = 0$.  In this case, $\Gamma_S(\pi_S(M))$ is determined by being maximal among all extensions of $M$ whose cokernel has singular support in $Z(S)$ and that contain no nonzero submodule with singular support in $Z(S)$.  Consider the natural injective map $\iota: M\rightarrow (Q_S^\mu(\alg)\otimes_{\alg} M)^{\on{rat}}$.  If $M$ is finitely generated over $\alg$, then 
Lemma \ref{local basics}(5) implies that $SS(\on{coker}(\iota))\subseteq Z(S)$; since the target of $\iota$ commutes with colimits in $M$, the same inclusion of singular support holds for arbitrary $M$.  Now suppose $f: M\rightarrow M'$ is any map in $(\alg, G,c)-\on{mod}$ whose cokernel has singular support in $Z(S)$.  Then the natural map $Q_S^\mu(\alg)\otimes_{\alg} M\rightarrow Q_S^\mu(\alg)\otimes_{\alg} M'$ is an isomorphism, so the same is true after passing to rational vectors.  We thus obtain a map 
$g: M'\rightarrow (Q_S^\mu(\alg)\otimes_{\alg} M)^{\on{rat}}$ so that $g\circ f = \iota$.  Passing to the colimit over such $M'$, we immediately obtain a map
$\Gamma_S(\pi_S(M))\rightarrow (Q_S^\mu(\alg)\otimes_{\alg} M)^{\on{rat}}$.  Since neither domain nor target has a submodule with singular support in $Z(S)$, the universal property of $\Gamma_S$ implies that this map is an isomorphism.  

To prove the exactness claim of (3), observe that $\on{Ind}_\alg^{Q_S^\mu(\alg)}$ is exact.  Hence it suffices to prove that $(-)^{\on{rat}}$ is exact on the image of $\on{Ind}_\alg^{Q_S^\mu(\alg)}$.  Left exactness is clear.  To check right exactness, we may assume that $M$ and $N$ are finitely generated $\alg$-modules and that $M\rightarrow N$ is surjective; moreover, it suffices to check for a surjective map $\cR(M)\rightarrow \cR(N)$ corresponding to good filtrations of $M$ and $N$.  Now, by the construction of \cite{AVV}, the map 
$\wt{Q}_{\wt{S}}^\mu(\alg)\otimes \cR(M)\rightarrow \wt{Q}_{\wt{S}}^\mu(\alg)\otimes \cR(N)$ is the limit of localizations of the maps $\cR(M)/t^k\cR(M)\longrightarrow \cR(N)/t^k\cR(N)$, which are surjective maps of rational $G$-modules by construction; for simplicity, write $M_k\rightarrow N_k$ for the maps of localized modules.  A vector $v \in \wt{Q}_{\wt{S}}^\mu(\alg)\otimes\cR(N)$ is a rational vector if and only if there are a finite-dimensional rational $G$-module $V$ and vector $\wt{v}\in V$ and a compatible sequence of $G$-maps $V\rightarrow N_k$ for all $k$ such that the image of $\wt{v}$ corresponds to $v$ in the inverse limit.  Since $G$ is reductive and the maps $M_{k+1}\rightarrow M_k$ and $M_k \rightarrow N_k$ are all surjective maps of rational $G$-modules, given such a sequence we can find a compatible sequence of $G$-equivariant lifts $V\rightarrow M_k$ whose composites to $N_k$ agree with the given maps, and such that the images of $\wt{v}$ in the $M_k$ form a compatible sequence of vectors.  Hence, in the limit, we get a lift of $v$ to a rational vector in $\ds\lim_{\longleftarrow} M_k = \wt{Q}_{\wt{S}}^\mu(\alg)\otimes\cR(M)$ (here we use that $\cR(M)$ is finitely generated, otherwise the microlocalization may not equal the tensor product).  
\end{proof}

\begin{corollary}
\label{j_S^*}
Suppose that the set $S$ as above satisfies $\ds \bigvar^{uns} \subseteq \bigcup_{s\in S} V(s)$ (where instability is taken with respect to $\chi$).  Then 
the functors $\pi_S$ factor canonically
\bd
\xymatrix{(\alg,G,c)-\on{mod} \ar[d]^{\pi_c} \ar[dr]^{\pi_S} \\
\E_{\resol}(c)-\on{mod} \ar[r]^{\hspace{-1em}j_S^*} & \ds\frac{(\alg, G, c)-\on{mod}}{(\alg, G, c)-\on{mod}_{Z(S)}} 
} \hspace{2em}
\xymatrix{(\cR,G,c)-\on{mod} \ar[d]^{\pi_c} \ar[dr]^{\pi_S} \\
\ds\frac{(\cR,G,c)-\on{mod}}{(\cR,G,c)-\on{mod}^{\on{uns}}} \ar[r]^{\hspace{-1em}j_S^*} & \ds\frac{(\cR(\alg), G, c)-\on{mod}}{(\cR(\alg), G, c)-\on{mod}_{Z(S)}} 
}
\ed
through exact functors $j_S^*$.
\end{corollary}
Note that here (and elsewhere in this section) $\cR-\on{mod}$ denotes the category of graded $\cR$-modules (and similarly for the equivariant categories). 
\subsection{Adjoints}
\begin{prop}\label{microlocal adjoint}
The functors $j_S^*$ of Corollary \ref{j_S^*} have right adjoints.
More precisely,
\begin{enumerate}
\item the adjoints $j_{S\ast}$ are given by
$j_{S\ast}M = \pi_c\circ\Gamma_S(M).$
\item Each $j_{S\ast}$ is exact.
\end{enumerate}
\end{prop}
\begin{proof}
Note that if an adjoint exists, then $\Gamma_c\circ j_{S\ast} = \Gamma_S$ by uniqueness of adjoints, and so
$j_{S\ast} = \pi_c\circ\Gamma_c\circ j_{S\ast} = \pi_c\circ\Gamma_S$; and similarly for $\cR$.  

We give the proof only for $\alg$; the proof for $\cR$ is similar.  We will show that the given formula defines a right adjoint.  

We begin by proving:
\begin{lemma}\label{no change in push-pull}
If $\wt{N} = \Gamma_S\pi_S(N)$, where $N$ is a $(G,c)$-equivariant $\alg$-module, then 
$\Gamma_c\circ\pi_c(\wt{N}) = \wt{N}$.  Similar statements hold for weakly equivariant and non-equivariant modules.
\end{lemma}
\begin{proof}
Since $\pi_S$ is a localization functor, we have $\pi_S\circ\Gamma_S = \on{Id}$.  It follows that $\Gamma_S\circ\pi_S(\wt{N}) = \wt{N}$.
Now suppose that $f: \wt{N} \rightarrow N'$ is any map whose kernel and cokernel have unstable singular support.  Then 
$\pi_S(\wt{N})\rightarrow \pi_S(N')$ becomes an isomorphism, and so adjunction defines a map $g: N'\rightarrow \Gamma_S\circ\pi_S(\wt{N}) = \wt{N}$ so that $g\circ f$ is the identity.  Now $\Gamma_c\circ\pi_c(\wt{N})$ is the colimit over such maps $f$, and since split maps do not contribute to the colimit, the conclusion follows.
\end{proof}

Returning to the proof of the proposition, we thus have 
\begin{align*}
\Hom_{\E_{\resol}(c)}\big(\pi_c M, \pi_c\Gamma_S(\pi_S(N))\big) & = \Hom_{(\alg, G, c)}\big(M, \Gamma_c\pi_c\Gamma_S\pi_S(N)\big)\\
& = \Hom_{(\alg,G,c)}\big(M, \Gamma_S\pi_S(N)\big) \;\; \text{by Lemma \ref{no change in push-pull}}\\
& = \Hom\big(\pi_S(M),\pi_S(N)\big)\\
& =  \Hom\big(j_S^* \pi_c(M),\pi_S(N)\big).
\end{align*}
Since $\pi_c$ and $\pi_S$ are essentially surjective, this proves (1).  The exactness statement (2) is immediate from (1) by Proposition \ref{right adj of pi_S}(3).  
This completes the proof of Proposition \ref{microlocal adjoint}.
\end{proof}

\begin{prop}\label{computing jS}
Let $M = j_{S\ast} \pi_S(N)$ where $N$ is a twisted equivariant $\alg$-module or $\cR$-module.  Then 
\bd
\mathbb{R}\Gamma_c(M) \cong (Q_S^\mu(\alg)\otimes_{\alg} N)^{\on{rat}},
\hspace{2em} \text{respectively}\hspace{2em}
\mathbb{R}\Gamma_c(M) \cong (\wt{Q}_{\wt{S}}^\mu(\alg)\otimes_{\alg} N)^{\on{rat}}.
\ed
\end{prop}
\begin{proof}
Since $j_{S\ast}$ and $\pi_S$ are exact, we have 
\bd
(\mathbb{R}\Gamma_c)\circ (j_{S\ast}\pi_S) = \mathbb{R}(\Gamma_c\circ j_{S\ast}\circ \pi_S) = \mathbb{R}(\Gamma_c\pi_c\Gamma_S\pi_S).
\ed
But Lemma \ref{no change in push-pull} implies that $\Gamma_c\pi_c\Gamma_S\pi_S = \Gamma_S\pi_S$, and the latter is exact by 
Proposition \ref{right adj of pi_S}, so we get $(\mathbb{R}\Gamma_c)\circ (j_{S\ast}\pi_S) = \Gamma_S\pi_S$, as required.  
\end{proof}

\subsection{\v{C}ech Complex}
Choose a finite collection $\{f_i \; |\; 0\leq i\leq n\}$ of nonzero homogeneous $G$-semi-invariants, so that 
$\ds\bigcup_{i=0}^n D(f_i) = \bigvar^{ss}$.  For each nonempty subset $I\subseteq \{0,\dots,n\}$ we let $S_I = \{f_i\; | \; i\in I\}$.  For any weakly $G$-equivariant $\alg$-module $M$, we get
\bd
M_I \overset{\on{def}}{=} \Gamma_{S_I}\pi_{S_I}(M) = \big(Q_{S_I}^{\mu}(\alg)\otimes_\alg M\big)^{\on{rat}}.
\ed
Given $I\subseteq J\subseteq \{0,\dots,n\}$, we get a map $M_I\rightarrow M_J$ induced by the natural map $Q_{S_I}^\mu(\alg)\rightarrow Q_{S_J}^\mu(\alg)$.  We form the {\em \v{C}ech complex} $\check{C}^\bullet(M)$ with
\bd
\check{C}^k(M) = \prod_{|I| = k+1} M_I,
\ed
and differentials defined as in \cite[Section~III.4]{H}.  
If $M$ lies in 
$(\alg,G,c)-\on{mod}$, then $\check{C}^\bullet(M)$ is a complex in the same category.  We make analogous definitions for graded $\cR$-modules $M$ using $\wt{Q}_{\wt{S}_I}^\mu(\alg)$ in place of $Q^\mu_{S_I}(\alg)$ and let $\check{C}^\bullet(M)$ denote the corresponding \v{C}ech complex.

\begin{prop}\label{Cech of a quotient}
For a $G$-invariant left ideal $I\subseteq \alg$, $\check{C}^\bullet(\alg/I) \cong \check{C}^\bullet(\alg)\otimes_{\alg} (\alg/I)$.
\end{prop}
\begin{proof}
Picking generators of the ideal $I$ and using the fact that $\alg$ is a rational $G$-representation, we may find a finite-dimensional rational representation $V$ of $G$ and an equivariant exact sequence $\alg\otimes_{\C} V \rightarrow \alg \rightarrow \alg/I\rightarrow 0$.  Tensoring with $Q_S^\mu(\alg)$ over $\alg$ and applying $(-)^{\on{rat}}$ for any $S$  gives an exact sequence (Proposition \ref{right adj of pi_S})
$\big(Q_S^\mu(\alg)\otimes_{\C} V\big)^{\on{rat}} \rightarrow Q_S^\mu(\alg)^{\on{rat}} \rightarrow \big(Q_S^\mu(\alg)\otimes_{\alg} \alg/I\big)^{\on{rat}}\rightarrow 0$.  Now by Lemma \ref{rational vectors of a free module} the first term of the sequence is just $Q_S^\mu(\alg)^{\on{rat}}\otimes_{\C} V$, and hence its image in $Q_S^\mu(A)^{\on{rat}}$ is $Q_S^\mu(A)^{\on{rat}}I$. It follows that
\bd
(Q_S^\mu(\alg)\otimes_{\alg} \alg/I)^{\on{rat}} \cong Q_S^\mu(\alg)^{\on{rat}}/Q_S^\mu(A)^{\on{rat}}I \cong
Q_S^\mu(\alg)^{\on{rat}}\otimes_{\alg} (\alg/I).
\ed
 The result is then immediate from the definition of the \v Cech complex.
\end{proof}

\subsection{Rees Modules}\label{Rees modules}
Suppose now that $M$ is an $\alg$-module equipped with a good filtration.  We obtain the associated Rees module 
$\cR(M) = \oplus_k F_k(M)t^k \subseteq M[t, t^{-1}]$.  This is an $\cR = \cR(\alg)$-module; it is always torsion-free, hence flat, over $\C[t]$.  Each microlocalization $Q_S^\mu(\alg)\otimes_\alg M$
comes equipped with a filtration as well.  Moreover, we have the following:
\begin{prop}\label{Rees Cech}
Suppose that $M$ is an $\alg$-module with good filtration.  Then:
\mbox{}
\begin{enumerate}
\item $\check{C}^\bullet\big(\cR(M)\big)$ is a complex of torsion-free, hence flat, $\C[t]$-modules.
\item $\check{C}^\bullet\big(\cR(M)\big)/(t-a)\cong \check{C}^\bullet(M)$ as $\cR/(t-a)\cR \cong \alg$-modules for every $a\neq 0$.
\item $\check{C}^\bullet\big(\cR(M)\big)/(t)\cong \check{C}^\bullet(\on{gr}(M))$ as $\C[\bigvar]$-modules.
\end{enumerate}
\end{prop}
Here part (3) follows from Lemma \ref{local basics}(5). 
\begin{corollary}\label{gr commutes with H^0}
For any $(\alg, G)$-module $P$ with good filtration and any $i$, there is a natural $G$-equivariant isomorphism
 \bd
 \on{gr}\big(H^i(\check{C}^\bullet P)\big) \cong H^i((\check{C}^\bullet(\on{gr}\, P)\big)
\ed 
that is functorial in $P$.
\end{corollary}
\begin{proof}
This follows from the standard fact that cohomologies of complexes of flat modules commute with base change.
\end{proof}
We record the following two lemmas that will be used in Section \ref{SD proof}.
\begin{lemma}[Lemma III.12.3 of \cite{H}]\label{replacement lemma}
Let $R$ be a Noetherian ring and let $C^\bullet$ be a bounded-above complex of left $R$-modules such that for each $i$, $H^i(C^\bullet)$ is a 
finitely generated $R$-module.  Then there is a bounded-above complex $L^\bullet$ of finitely generated free $R$-modules and a quasi-isomorphism 
$g: L^\bullet \rightarrow C^\bullet$.  Furthermore, if $S\subseteq R$ is a central subring of $R$ and all $C^i$ are flat over $S$, then for any $S$-module $M$, the induced map $g\otimes 1_M: L^\bullet \otimes M\rightarrow C^\bullet\otimes M$ is a quasi-isomorphism.
\end{lemma}

\begin{lemma}\label{base change of Hom}
Let $\cR$ be a non-negatively graded, Noetherian, flat $\C[t]$-algebra (in particular, $\C[t]$ is central in $\cR$).  Suppose that $M$ is a graded $\cR$-module of finite type and $N^\bullet$ is a complex of $\C[t]$-flat graded $\cR$-modules.  Then for any $a\in \C$,
\bd
\C[t]/(t-a)\otimes_{\C[t]}\Hom_{\cR}(M, N^\bullet) \cong \Hom_{\cR/(t-a)\cR}(M/(t-a)M, N^\bullet/(t-a)N^\bullet).
\ed
\end{lemma}

\subsection{Calculating with the \v{C}ech Complex}
\begin{thm}\label{Cech calculates}
For every $M$ in 
$(\alg,G,c)-\on{mod}$, respectively $(\cR,G,c)-\on{mod}$, we have
\begin{enumerate}
\item $\pi_c M \simeq \pi_c \check{C}^\bullet(M)$, and
\item $\mathbb{R}\Gamma_c\circ\pi_c(M)\simeq \check{C}^\bullet(M)$.
\end{enumerate}
\end{thm}
\begin{proof}
It suffices to prove (1) for finitely generated modules $M$ since both sides commute with colimits.  If $M$ is an $\alg$-module, it can be equipped with an equivariant good filtration to give a finitely generated equivariant graded $\cR$-module $\wt{M}$ for which $\wt{M}/(t-1)\wt{M} = M$.  
Since the functor $\C[t](t-1)\otimes -$ descends to the quotient categories and 
$\C[t](t-1)\otimes \big(\check{C}^\bullet(\wt{M})\big) \simeq \check{C}^\bullet(M)$ (by Proposition \ref{Rees Cech}), it will suffice to show that for a finitely generated $\cR$-module $M$, the statement of (1) holds.
 
Let $\cR-\on{mod}^{uns}$ denote the full subcategory of objects with unstable singular support.  We have a commutative diagram:
\bd
\xymatrix{\ds\frac{(\cR,G,c)-\on{mod}}{(\cR,G,c)-\on{mod}^{\on{uns}}} \ar[r]^{\mbox{}\hspace{1em}\on{incl}} & \ds\frac{\cR-\on{mod}}{\cR-\on{mod}^{uns}}\\
(\cR,G,c)-\on{mod}\ar[u]_{\pi_c} \ar[r]^{\on{incl}} & \cR-\on{mod}\ar[u]_{\overline{\pi}}
}
\ed
with faithful rows.  To check that the canonical map $\pi_c M \rightarrow \pi_c \check{C}^\bullet(M)$ is a quasi-isomorphism, then, it is enough to check after applying the functor
$\on{incl}$, hence to check that $\overline{\pi} M\simeq \overline{\pi} \check{C}^\bullet(M)$.  Choose a finite, free graded $\cR$-module resolution $F^\bullet\rightarrow M$.  We get a map 
$\overline{\pi} F^\bullet \rightarrow \overline{\pi} \check{C}^\bullet(F^\bullet)$, and if this is a quasi-isomorphism (replacing the second double complex by its totalization) then the conclusion follows for $M$.  Since both $\overline{\pi}$ and $\check{C}^\bullet$ commute with colimits, the quasi-isomorphism for $F^\bullet$ reduces to the corresponding statement for $\cR$ itself.  

We have $\C[t]/(t)\otimes_{\C[t]}\check{C}^\bullet(\cR) \cong \check{C}^\bullet(\on{gr}(\alg))$ and thus the associated graded of $\cR\rightarrow \check{C}^\bullet(\cR)$ is the natural map 
$\C[\bigvar]\rightarrow \check{C}^\bullet(\C[\bigvar])$.  The target of this last map is a complex of $\C[\bigvar]$-modules that computes
$H^\bullet(\bigvar^{ss}, \theo)$.  Since $\bigvar$ is affine we have 
\bd
\on{supp}(H^i(\bigvar^{ss}, \theo))\subseteq \bigvar^{uns} \;\;\;\; \text{for $i>0$, implying}
\ed
\bd
SS\big(\on{Cone}[\cR\rightarrow \check{C}^\bullet (\cR)]\big) \subseteq \bigvar^{uns}.
\ed
It follows that $\overline{\pi}\cR\xrightarrow{\simeq} \overline{\pi}\check{C}^\bullet(\cR)$ is a quasi-isomorphism, proving (1). 

For part (2), we have:
$\mathbb{R}\Gamma_c(\pi_c(M)) \simeq \mathbb{R}\Gamma_c(\pi_c\check{C}^\bullet(M))\simeq \check{C}^\bullet(M)$,
where the first isomorphism comes from part (1) of the present theorem and the second isomorphism follows from 
Proposition \ref{computing jS} and
Proposition \ref{microlocal adjoint}(1).
\end{proof}

\bibliographystyle{alpha}

\begin{thebibliography}{BLPW2}

\bibitem[AV]{AV} M Awami and F. Van Oystaeyen, On filtered rings with Noetherian associated graded rings, in {\em Ring theory, Granada, 1986}, Springer Lecture Notes in Mathematics, vol. 1328, 8--27, Springer Verlag, New York, 1988.

\bibitem[AVV]{AVV} M. Asensio, M. Van den Bergh, and F. Van Oystaeyen, A new algebraic approach to microlocalization of filtered rings, {\em Trans. Amer. Math. Soc.} {\bf 316} (1989), no. 2, 537--553.

\bibitem[Bea]{B} A. Beauville, Symplectic singularities, {\em Invent. Math.} {\bf 139} (2000), no. 3, 541-549.

\bibitem[BB]{BB} A. Beilinson and J. Bernstein, A proof of Jantzen conjectures, {\em I.M. Gel'fand Seminar}, 1--50,
{\em Adv. Soviet Math.} {\bf 16}, Part 1, Amer. Math. Soc., Providence, RI, 1993.

\bibitem[BD]{BD} A. Beilinson and V. Drinfeld, Quantization of Hitchin's integrable system and Hecke eigensheaves,
available at {\tt http://www.math.uchicago.edu/~mitya}.

\bibitem[BG]{BG} A. Beilinson and V. Ginzburg, Wall-crossing functors and $\D$-modules, 
{\em Represent. Theory} {\bf 3} (1999), no. 1, 1--31.

\bibitem[BKu]{BKu} G. Bellamy and T. Kuwabara, On deformation quantizations of hypertoric varieties,
{\tt arXiv:1005.4645}.

\bibitem[BFN]{BFN} D. Ben-Zvi, J. Francis, and D. Nadler, Integral transforms and Drinfeld centers in derived algebraic geometry,
{\em J. Amer. Math. Soc.} {\bf 23} (2010), no. 4, 909--966.

\bibitem[BNa]{BNa} D. Ben-Zvi and D. Nadler, The character theory of a complex group, {\tt arXiv:0904.1247}.

\bibitem[BL]{BL} J. Bernstein and V. Lunts, Localization for derived categories of $(\mathfrak{g}, K)$-modules, {\em J. Amer. Math. Soc.}
{\bf 8} (1995), no. 4, 819--856.

\bibitem[BK]{BK} R. Bezrukavnikov, D. Kaledin, Fedosov quantization in the algebraic context, {\em Moscow Math. J.} {\bf 4} (2004), no. 3, 559--592.

\bibitem[BK1]{BK1} R. Bezrukavnikov, D. Kaledin, McKay equivalence for symplectic resolutions of quotient singularities, {\em Proc. Steklov Inst. Math.} {\bf 246} (2004), no. 3, 13--33.

\bibitem[BMR]{BMR} R. Bezrukavnikov, I. Mirkovic, and D. Rumynin, Localization of modules for a semisimple Lie algebra in prime characteristic,
{\em Ann. of Math.} (2) {\bf 167} (2008), no. 3, 945--991.

\bibitem[Bj1]{Bjbook} J.-E. Bj\"ork, {\em Rings of differential operators}, North-Holland, Amsterdam, 1979.

\bibitem[Bj2]{Bj} J.-E. Bj\"ork, The Auslander condition on Noetherian rings, in {\em S\'eminaire d'Alg\`ebre Dubreil-Malliavin}, Lecture Notes in Mathematics, vol. 1404, 137--173, Springer Verlag, New York, 1989. 

\bibitem[B\"oN]{BN} M. B\"okstedt and A. Neeman, Homotopy limits in triangulated categories, {\em Compositio Math.} {\bf 86} (1993), no. 2, 209--234.

\bibitem[BLPW]{BLPW} T. Braden, A. Licata, N. Proudfoot, and B. Webster, Hypertoric category $\theo$, {\tt arXiv:1010.2001}.

\bibitem[BPW]{BLPW2} T. Braden, N. Proudfoot, and B. Webster, Quantizations of conical symplectic resolutions I: local and global structure, {\tt arXiv:1208.3863}.

\bibitem[Br]{Bridgeland} T. Bridgeland, Equivalences of triangulated categories and Fourier-Mukai transforms {\em Bull. London Math. Soc. } {bf 31} (1999), 25--34.


\bibitem[CB]{CB-flat} W. Crawley-Boevey, Geometry of the moment map for representations of quivers,
{\em Compositio Math.} {\bf 126} (2001), no. 3, 257--293.

\bibitem[Dr]{Drinfeld} V. Drinfeld, DG quotients of DG categories, {\em J. Algebra} {\bf 272} (2004), no. 2, 643--691.

\bibitem[DG]{DG} C. Dunkl and S. Griffeth, Generalized Jack polynomials and the representation theory of rational Cherednik algebras,
{\em Selecta Math. (N.S.)} {\bf 16} (2010), 791--818.

\bibitem[Ei]{Ei} D. Eisenbud, {\em Commutative algebra with a view toward algebraic geometry}, Springer-Verlag, New York, 1995.

\bibitem[Ek]{Ek} E. Ekstr\"om, The Auslander condition on graded and filtered Noetherian rings, in 
{\em S\'eminaire d'Alg\`ebre Dubreil-Malliavin}, Lecture Notes in Mathematics, vol. 1404, 220--245, Springer Verlag, New York, 1989.

\bibitem[Et1]{Et1} P. Etingof,  Cherednik and Hecke algebras of varieties with a finite group action,
{\tt arXiv:math/0406499}.

\bibitem[Et2]{E} P. Etingof, Symplectic reflection algebras and affine Lie algebras, preprint, arXiv:1011.45884

\bibitem[EGGO]{EGGO} P. Etingof, W.-L. Gan, V. Ginzburg, and A. Oblomkov, 
Harish-Chandra homomorphisms and symplectic reflection algebras for wreath products,
{\em Publ. Math. Inst. Hautes \'Etudes Sci.} {\bf 105} (2007), 91--155.

\bibitem[EG]{EG} P. Etingof and V. Ginzburg, Symplectic reflection algebras, Calogero-Moser space, and
deformed Harish-Chandra homomorphism, {\em Invent. Math.} {\bf 147} (2002), no. 2, 243--348.

\bibitem[FG]{FG} M. Finkelberg and V. Ginzburg, Cherednik algebras for algebraic curves,
in {\em Representation theory of algebraic groups and quantum groups}, 121--153, {\em Birkh\"auser/Springer}, New York, 2010.


\bibitem[Ga]{Gabriel} P. Gabriel, Des cat\'egories ab\'eliennes, {\em Bull. Soc. Math. France} {\bf 90} (1962), 323--448.

\bibitem[GG]{GG} W. L. Gan and V. Ginzburg, Almost-commuting variety,
$\D$-modules, and Cherednik algebras, With an appendix by Ginzburg, {\em IMRP Int. Math. Res. Pap.} (2006),
 1--54.

\bibitem[GW]{GW} K. Goodearl and R. Warfield, {\em An introduction to noncommutative Noetherian rings}, London Mathematical Society, Cambridge, 1989.

\bibitem[GL]{GL} I. Gordon and I. Losev, On category $\theo$ for cyclotomic rational Cherednik algebras, {\tt arXiv:1109.2315}.

\bibitem[GS1]{GS1} I. Gordon and J. T. Stafford, Rational Cherednik algebras and Hilbert schemes,
{\em Adv. Math.} {\bf 198} (2005), no. 1, 222--274.



\bibitem[GS2]{GS2} I. Gordon and J. T. Stafford, Rational Cherednik algebras and Hilbert schemes II: representations
and sheaves, {\em Duke Math. J. } {\bf 132} (2006), no. 1, 73--135.

\bibitem[Go]{Gordon} I. Gordon, A remark on rational Cherednik algebras and differential operators on the 
cyclic quiver, {\em Glasg. Math. J.} {\bf 48} (2006), no. 1, 145--160.

\bibitem[Go2]{Gordon2} I. Gordon, Quiver varieties, category $\theo$ for rational Cherednik algebras, and 
Hecke algebras, {\em Int. Math. Res. Pap.} {\bf 2008} (2008), Art. ID rpn006, 69 pages, doi:10.1093/imrp/rpn006.

\bibitem[H]{H} R. Hartshorne, {\em Algebraic Geometry}, Springer-Verlag, 1977.

\bibitem[Ho]{Holland} M. Holland, Quantization of the Marsden-Weinstein reduction for extended Dynkin quivers,
{\em Ann. Sci. \'Ecole Norm. Sup. (4)} {\bf 32} (1999), no. 6, 813--834.

\bibitem[HTT]{HTT} R. Hotta, K. Takeuchi, and T. Tanisaki, {\em $\mathcal{D}$-Modules, Perverse Sheaves, and Representation 
Theory}, Birkha\"user, 2008.

\bibitem[HK]{HK} Y. Hu and S. Keel, Mori dream spaces and GIT, {\em Michigan Math. J.} {\bf 48} (2000), no. 1, 331--348.


\bibitem[Ka1]{Kashiwara1} M. Kashiwara, {\em $\mathcal{D}$-Modules and Microlocal Calculus}, Amer. Math. Soc., 2003.

\bibitem[Ka2]{Kashiwara} M. Kashiwara, Equivariant derived category and representations of real semisimple Lie groups, in {\em Representation theory and complex analysis}, Lecture Notes in Mathematics 1931, Springer-Verlag, 2008.


\bibitem[KR]{KR} M. Kashiwara and R. Rouquier, Microlocalization of rational Cherednik algebras, {\em Duke 
Math. J.} {\bf 144} (2008), no. 3, 525--573.

\bibitem[KS]{KS} M. Kashiwara and P. Schapira, {\em Categories and sheaves}, Springer-Verlag, 2006. 

\bibitem[KS2]{KSDQ} M. Kashiwara and P. Schapira, Deformation quantization modules, {\em Ast\'erisque} {\bf 345} (2012).

\bibitem[Ki]{King} A. King, Moduli of representations of finite-dimensional algebras, {\em Quart. J. Math. Oxford Ser. (2)} {\bf 45} (1994), 515--530.

\bibitem[Li]{Li} H. Li, Lifting Ore sets of Noetherian filtered rings and applications, {\em J. Algebra} {\bf 179} (1996), 686--703.

\bibitem[LvO1]{LVO} H. Li and F. van Oystaeyen, Zariskian filtrations, {\em Comm. Alg.} {\bf 17} (1989), no. 12, 2945--2970.

\bibitem[LvO2]{LvO} H. Li and F. van Oystaeyen, {\em Zariskian filtrations}, Kluwer Academic Publishers, Dordrecht, 1996.

\bibitem[LN]{LN} J. Lipman and A. Neeman, Quasi-perfect scheme-maps and boundedness of the twisted inverse image functor, 
{\em Illinois J. Math.} {\bf 51} (2007), no. 1, 209--236.

\bibitem[Lu]{Lu} J. Lurie, Derived algebraic geometry II: noncommutative algebra, {\tt arXiv:math/0702299}.

\bibitem[MR]{MR} J. McConnell and J. Robson, {\em Noncommutative Noetherian rings}, corrected printing,  American Mathematical Society, 2001.

\bibitem[MV]{MV} I. Musson and M. Van den Bergh, Invariants under tori of rings of differential operators and related topics, {\em Mem. Amer. Math. Soc.} {\bf 136} (1998), no. 650.

\bibitem[NV]{NV} C. N\u{a}st\u{a}sescu and F. Van Oystaeyen, {\em Graded and filtered rings and modules}, Lecture Notes in Mathematics 758, Springer-Verlag, 1979.

\bibitem[Ne]{Neeman} A. Neeman, The Grothendieck duality theorem via Bousfield's techniques and Brown representability,
{\em J. Amer. Math. Soc.} {\bf 9} (1996), no. 1, 205--236.

\bibitem[O]{O} A. Oblomkov, Deformed Harish-Chandra homomorphism for the cyclic quiver, 
{\em Math. Res. Lett.} {\bf 14} (2007), no. 3, 359--372.

\bibitem[Po]{Popescu} N. Popescu, {\em Abelian categories with applications to rings and modules,} Academic Press, New York, 1973. 


\bibitem[R]{R} R. Rouquier, \textit{$q$-Schur algebras and complex reflection groups}, Moscow Math. J. Vol. 8, no. 1, 119--158.


\bibitem[Sch]{Sch} G. Schwarz, Lifting differential operators from orbit spaces, 
{\em Ann. Sci. \'Ecole Norm. Sup. (4)} {\bf 28} (1995), no. 3, 253--305.

\bibitem[SS]{SS} S. Schwede and B. Shipley, Stable model categories are categories of modules, {\em Topology} {\bf 42} (2003), no. 1, 103--153.

\bibitem[T]{T} C. Teleman, The quantization conjecture revisited, {\em Ann. of Math. (2)} {\bf 152} (2000), no. 1, 1--43.

\bibitem[VdB]{VdB} M. Van den Bergh, Some generalities on $G$-equivariant quasicoherent $\theo_X$ and $\D_X$-modules, 
{\tt http://hardy.uhasselt.be/personal/vdbergh/Publications/Geq.ps}.


\bibitem[VW1]{VW} F. Van Oystaeyen and L. Willaert, Grothendieck topology, coherent sheaves, and Serre's
theorem for schematic algebras, {\em J. Pure Appl. Algebra} {\bf 104} (1995), 109--122.

\bibitem[VW2]{VW2} F. Van Oystaeyen and L. Willaert, Cohomology of schematic algebras,
{\em J. Algebra} {\bf 185} (1996), 74--84.

\bibitem[We]{Weibel} C. Weibel, {\em An Introduction to Homological Algebra}, Cambridge UP, Cambridge, 1994.

\bibitem[YZ1]{YZ} A. Yekutieli and J. Zhang, Rings with Auslander dualizing complexes, {\em J. Algebra} {\bf 213} (1999), no. 1, 1--51.

\bibitem[YZ2]{YZ-differential} A. Yekutieli and J. Zhang, Dualizing complexes and perverse modules over
differential algebras, {\em Compositio Math.} {\bf 141} (2005), no. 3, 620--654.

\end{thebibliography}

\end{document}